\documentclass{amsart}

\usepackage{amssymb,color}
\usepackage{amsfonts}
\usepackage{amsmath}
\usepackage{euscript}
\usepackage{enumerate}
\usepackage{pdfsync}

\newtheorem{theorem}{Theorem}[section]
\newtheorem{lemma}[theorem]{Lemma}

\newtheorem{prop}[theorem]{Proposition}
\newtheorem{cor}[theorem]{Corollary}

\newtheorem*{Theorem1'}{Theorem 1'}

\theoremstyle{definition}

\theoremstyle{remark}

\newcommand \Z{{\mathbb Z}}

\newcommand \N{{\mathbb N}}

\begin{document}

\title[A study of the Wamsley group and its Sylow subgroups]{A study of the Wamsley group and its Sylow subgroups}

\author{Andrea Previtali}
\address{Dipartimento di Matematica e Applicazioni, University of Milano-Bicocca, Milano, Italy}
\email{andrea.previtali@unimib.it}

\author{Fernando Szechtman}
\address{Department of Mathematics and Statistics, University of Regina, Canada}
\email{fernando.szechtman@gmail.com}
\thanks{The second author was partially supported by NSERC grant RGPIN-2020-04062}



\begin{abstract} We study the Wamsley group $\langle X,Y,Z\,|\, X^Z=X^\alpha, {}^Z Y=Y^\beta, Z^\gamma=[X,Y]\rangle$
and its Sylow subgroups, where $\alpha^\gamma\neq 1\neq \beta^\gamma$ and $\gamma>0$, obtaining the sharpest
results when $\alpha=\beta$.
\end{abstract}

\maketitle

\section{Introduction}\label{sec1}

A group presentation is balanced if it has the same number of generators as relations.
Finite groups that admit a finite balanced presentation have attracted considerable attention 
since Mennicke's discovery \cite{Me}, in 1959, of the first family of finite groups requiring 3 generators and 3 relations.
He found the groups $M(a,b,c)=\langle x,y,z\,|\, x^y=x^a, y^{z}=y^b, z^x=z^c\rangle$ 
to be finite when $a=b=c\geq 2$. Abelianizing yields that
$M(a,b,c)$ does require 3 generators when $a-1,b-1,c-1$ share a prime factor. 
In general, the finiteness of $M(a,b,c)$ was studied by Macdonald and Wamsley \cite{W}, Schenkman~\cite{S}, and Jabara~\cite{Ja}. 
A sufficient condition turns out to be $a,b,c\notin\{-1,1\}$. Upper bounds for the order of 
$M(a,b,c)$ were given by Johnson and Robertson \cite{JR}, Albar and Al-Shuaibi \cite{AA}, and Jabara \cite{Ja}.
In spite of having been carefully examined, the order of $M(a,b,c)$ is still not known in general. See \cite{Me,A,AA,Ja}
for details.

Since Mannicke's time, other 3-generator 3-relation finite groups were found \cite{W2,W3,P,J,Jam,Jam2, Al},
as well as 2-generator 2-relation finite groups \cite{M, W3, W4, CR, R, CRT, K, AS, AS2}. Perusal
of the literature indicates that, as in the case of $M(a,b,c)$, finite groups with a finite balanced presentation tend to not yield easily to a structural analysis.

In this paper, we study the Wamsley group \cite{W2}
$$
W=\langle X,Y,Z\,|\, X^Z=X^\alpha, {}^Z Y=Y^\beta, Z^\gamma=[X,Y]\rangle,
$$
and its Sylow subgroups, assuming $\gamma>0$ and $\alpha^{\gamma}\neq 1\neq \beta^{\gamma}$, obtaining definite
results when $\alpha=\beta$. In this case, for each prime number $p\mid |W|$, 
we obtain the following features of a Sylow $p$-subgroup of $W$:
its order; a presentation; a normal form for its elements; its nilpotency class. As a result, we deduce the order of $W$.
Our local work also yields that global derived length of~$W$. 

Our method consists of approaching $W$ via the Macdonald group \cite{M}
$$
G=\langle A,B\,|\, A^{[A,B]}=A^{\alpha^\gamma}, {}^{[A,B]} B=B^{\beta^\gamma} \rangle,
$$
by means of the homomorphism $\Omega:G\to W$, given by
\begin{equation}
\label{omega}
A\mapsto X,\; B \mapsto Y.
\end{equation}
The fact that the assignment (\ref{omega}) extends to a group homomorphism is ensured by the relations
$$
X^{[X,Y]}=X^{Z^\gamma}=X^{\alpha^\gamma},\; {}^{[X,Y]} Y={}^{Z^\gamma} Y=Y^{\alpha^\gamma}.
$$
Let $N$ stand for the kernel of $\Omega$. 
Since $G$ is finite by \cite{M}, it follows that $G^\Omega=\langle X,Y\rangle$ is a finite normal subgroup of 
$W$ with finite cyclic quotient $W/\langle X,Y\rangle$ of order $\delta$, where 
$\delta$ is a factor of~$\gamma$. Thus $W$ is finite and 
\begin{equation}
\label{order}
|W|=\delta\times |G|/|N|.
\end{equation}

In Theorem \ref{one}, we show that $\delta=\gamma$, that $N$ is the normal closure in $G$
of the $\gamma-1$ elements
\begin{equation}
\label{normalclo}
[A^\alpha,B^\mu][A,B]^{-1},\dots, [A^{\alpha^{\gamma-1}},B^{\mu^{\gamma-1}}][A,B]^{-1},
\end{equation}
where $\mu$ is inverse to $\beta$ modulo the order of $B$, and obtain a presentation of~$W$ which explicitly exhibits it
as an extension of $G/N$ by a cyclic group of order $\gamma$. Thus, by \cite[Section 4]{M}, the prime factors
of $|W|$ must be factors of $(\alpha^\gamma-1)(\beta^\gamma-1)\gamma$. Given a prime number $p$
dividing $(\alpha^\gamma-1)(\beta^\gamma-1)\gamma$, we study a Sylow $p$-subgroup $W_p$ of $W$.
The case when $p$ is not a factor of $(\alpha^\gamma-1)(\beta^\gamma-1)$ is trivial, so we assume
for the remainder of this section that $p\mid (\alpha^\gamma-1)(\beta^\gamma-1)$. Theorem \ref{dos} gives the analogue
of Theorem \ref{one} for $W_p$ and expresses it as an extension of $G_p/N_p$ with cyclic quotient of order $p^{v_p(\gamma)}$,
where $G_p$ is the Sylow $p$-subgroup of $G$ and $N_p=N\cap G_p$. It should be borne in mind that $G$ is nilpotent.
The proof of nilpotency given in \cite{M}
is incomplete, as explained in \cite{MS}. A full proof can be found in \cite{Sz}. 
Theorem \ref{tres} sharpens Theorem \ref{dos} when $p$ is a factor
of $(\alpha-1)(\beta-1)$. In particular, it proves that $N_p=1$ when $p|(\alpha-1)(\beta-1)$ but $p\nmid \gamma$.

Suppose for the remainder of this section that $\alpha=\beta$. In this case, a detailed description of $G_p$
is available to us \cite{MS}, which leaves $N_p$ and $G_p/N_p$ as the main obstacles to understand~$W_p$. 
Here $N_p$ is the normal closure in $G_p$ of the projections of the $\gamma-1$ elements (\ref{normalclo})
under the natural map $G\to G_p$. In order to make progress, we must explicitly compute 
the normal closure $N_p$ of these projections, and understand the structures of the quotient group $G_p/N_p$ and its
extension~$W_p$. These tasks
take up most of the paper. The behavior of $p$ is not uniform throughout this process, and we are forced to
consider three main cases as well as subcases thereof. Setting $m=v_p(\alpha^\gamma-1)$,
by Case 1 we mean that either $p>3$, or else $p=3$ with $m>1$ or $(\alpha^\gamma-1)/3\equiv 1\mod 3$.
By Case~2 we understand that $p=2$. By Case 3 we signify that $p=3$ and $\alpha^\gamma\equiv 7\mod 9$.


The commutator formulas needed to compute the projections of (\ref{normalclo}) can be found in \cite{MS2} in Cases 1 and 2.
In Case 1, an alternative derivation of this formula is given in Section~\ref{andrea}
by appealing to a recent algorithm due to Cant and Eick \cite{CE}. To apply this
algorithm, a knowledge of the factors of the lower central series of $G_p$ in Case 1 is needed.
This can be found in \cite[Theorem 8.1]{MS} and can alternatively be obtained by means of Witt's formula,
as shown in Section \ref{Witt}. Commutator formulas for Case~3 are derived in 
Section \ref{s6} with considerable effort, as $G_p$ is nilpotent of class 7 in this case.
It is only after these preliminary calculations that the determinations of $N_p$, $G_p/N_p$, and $W_p$ can be realized. This
is achieved in Section \ref{s4} for Case~1, see Theorems \ref{main1}, \ref{hzero}, and \ref{hzero2}; 
in Section \ref{s5} for Case 2, see Theorems \ref{main2} and \ref{main3}; and in Section~\ref{s7} for Case~3,
see Theorem~\ref{poop2}. Combining these results, we obtain the order of $W$ in Theorem~\ref{orderW}.
The nilpotency class of $W_p$ is given in Theorem \ref{nilp} and,
unlike the case of $G_p$, there is no upper bound for this class, in general. The solvability length of $W$ can
be found in Theorem \ref{solv}. 

In terms of notation, given a group  $T$ and $a,b\in T$, we set
$
a^b=a^{-1}ba,\; {}^a b=aba^{-1}.
$
For a fixed $b\in T$, the map $a\mapsto a^b$ is called conjugation by $b$. If $a_1,a_2,\dots\in T$, we define
$$
[a_1]=a_1\text{ and }[a_1,\dots,a_n]=[a_1,\dots,a_{n-1}]^{-1}a_n^{-1}[a_1,\dots,a_{n-1}]a_n\text{ for }n>1.
$$

\section{Approaching $W$ via $G$}\label{s2}

Since $A^{\alpha^\gamma}$ and $A$ (resp. $B^{\beta^\gamma}$ and $B$) 
are conjugate in $G$, it follows that $\alpha^\gamma$ (resp. $\beta^\gamma$)
is relatively prime to the order of $A$ (resp. $B$) and hence so is $\alpha$ (resp. $\beta$), and we fix
an integer $\mu$ that is inverse of $\beta$ modulo the order of $B$. A simple but useful observation is
the following identity in $G$:
$$
[A^{\alpha^\gamma},B^{\mu^\gamma}]=[A,B].
$$
This is so because conjugation by $[A,B]$ is an automorphism of $G$ that fixes $[A,B]$
and sends $A$ to $A^{\alpha^\gamma}$, $B$ to $B^{\mu^\gamma}$, and hence $[A,B]$ to $[A^{\alpha^\gamma},B^{\mu^\gamma}]$.

Recall the following well-known gadget (cf. \cite[Chapter III, Section 7]{Z}) to construct group extensions with finite cyclic factors.

\begin{theorem}\label{z} Let $T$ be an arbitrary group and $L$ a cyclic group of finite order $n\in\N$. Suppose that $t\in T$
and that $\Lambda$ is an automorphism of $T$ fixing $t$ and such that $\Lambda^n$ is conjugation by $t$. Then there
is a group $E$ containing $T$ as a normal subgroup, such that $E/T\cong L$, and for some $g\in E$ of order $n$ modulo $T$, we have 
$g^n=t$ and $\Lambda$ is conjugation by $g$.
\end{theorem}

Our next result shows that $|W|=\gamma |G|/|N|$,
describes $N$ as the normal closure of (\ref{normalclo}), and presents $W$ 
as an extension of $G/N$ by a cyclic group of order $\gamma$.

\begin{theorem}\label{one} Let $N_0$ be the normal closure in $G$ of the $\gamma-1$ elements (\ref{normalclo}).
Then
$$
N=N_0,\;\delta=\gamma,\;|W|=\gamma|G|/|N_0|,
$$
and $W$ is generated by elements  $a,b,d$ subject to the defining relations:
\begin{equation}
\label{defrel}
a^{[a,b]}=a^{\alpha^\gamma}, 
{}^{[a,b]} b=b^{\beta^\gamma}, [a^\alpha,b^\mu]=[a,b],\dots, [a^{\alpha^{\gamma-1}},b^{\mu^{\gamma-1}}]=[a,b],
\end{equation}
\begin{equation}
\label{defrel5}
a^d=a^\alpha, b^d=b^\mu\, (\text{that is}, {}^d b=b^\beta), d^\gamma=[a,b].
\end{equation}
\end{theorem}

\begin{proof} Conjugation by $Z$ is an automorphism, say $\Lambda$, of $W$. Since $Z^\gamma=[X,Y]$,
we see that $\Lambda$ fixes $[X,Y]$. But $X^\Lambda=X^\alpha$ and $Y^\Lambda=Y^\mu$, so repeated application of $\Lambda$ yields
$$
[X,Y]=[X^\alpha,Y^\mu]=[X^{\alpha^2},Y^{\mu^2}]=[X^{\alpha^3},Y^{\mu^3}]=\dots.
$$
This proves that $N_0$ is contained in $N$.

Consider the group $H=G/N_0$, the canonical projection $\pi:G\to H$, and set $a=A^\pi$, $b=B^\pi$. Then $H$
is generated by $a$ and $b$, with defining relations (\ref{defrel}).

From $[A^{\alpha^\gamma},B^{\mu^\gamma}]=[A,B]$, we deduce
$[a^{\alpha^\gamma},b^{\mu^\gamma}]=[a,b]$. The relations (\ref{defrel}) together with 
$[a^{\alpha^\gamma},b^{\mu^\gamma}]=[a,b]$ show that the assignment
\begin{equation}
\label{ass}
a\mapsto a^\alpha, b\mapsto b^\mu
\end{equation}
preserves all defining relations of $H$. Thus (\ref{ass}) 
extends to an endomorphism, say $\Pi$, of $H$. Since $\gcd(\alpha,o(a))=1=\gcd(\beta,o(b))$, it follows that $\Pi$
is an automorphism. As $[a^{\alpha},b^{\mu}]=[a,b]$, we see that $\Pi$ fixes $[a,b]$,
while the relations $a^{[a,b]}=a^{\alpha^\gamma}, 
{}^{[a,b]} b=b^{\beta^\gamma}$ show that $\Pi^\gamma$ is conjugation by $[a,b]$.
By Theorem~\ref{z}, there is a group $E=\langle a,b,d\rangle$ such that $\langle a,b\rangle$ is normal in $E$; $E/\langle a,b\rangle$
is cyclic of order $\gamma$; $d$ has order $\gamma$ modulo 
$\langle a,b\rangle$; and (\ref{defrel5}) holds. Clearly, the assignment
$$
X\mapsto a, Y\mapsto b, Z\mapsto d,
$$
extends to a group epimorphism $f:W\to E$. It follows that 
$
\gamma|G|/|N_0|\text{ is a factor of }|W|.
$
Thus by~(\ref{order}),
$
\gamma |G|/|N_0|\text{ is a factor of }
\delta |G|/|N|.
$
This implies that
$
\gamma |N| \text{ is a factor of }\delta |N_0|.
$
But we already know that 
$
\delta\mid\gamma\text{ and } |N_0|\mid|N|,
$
so
$\gamma=\delta,N_0=N$, and $f$ is an isomorphism.
\end{proof}

By \cite[Section 4]{M}, the prime factors of $|G|$ agree with those of $(\alpha^\gamma-1)(\beta^\gamma-1)$.
Thus, as a consequence of Theorem \ref{one}, every prime factor of $|W|$ must be a factor of $(\alpha^\gamma-1)(\beta^\gamma-1)\gamma$.


If $p\in\N$  is a prime that is not a factor of $(\alpha^\gamma-1)(\beta^\gamma-1)$, then writing
$\gamma=p^n k$, with $n\geq 0$ and $p\nmid k$, we see that a Sylow $p$-subgroup of $W$ is simply $\langle Z^k\rangle$ and has order $p^n$.

 
The group $G$ is not only finite but also nilpotent \cite{M,Sz}. Let $S$ be
the set of all primes $p\in\N$ that are factors of $(\alpha^\gamma-1)(\beta^\gamma-1)$. 
For $p\in S$, let $G_p$ and $N_p$ stand the Sylow $p$-subgroups of $G$ and $N$, respectively.
Then
$$
G=\underset{p\in S}\Pi G_p,\; N_p=G_p\cap N,\; N=\underset{p\in S}\Pi N_p,
$$
Likewise, for $p\in S$, we let $\langle X,Y\rangle_p$ 
be the Sylow $p$-subgroup of $\langle X,Y\rangle$, so that
$$
\langle X,Y\rangle=\underset{p\in S}\Pi \langle X,Y\rangle_p.
$$
Fix $p\in S$. We proceed to find presentations for $G_p$, $G_p/N_p$, and a Sylow $p$-subgroup $W_p$ of $W$.

We have $\gamma=p^n k$, where $n\geq 0$ and $p\nmid k$. Then $\gamma^k$ generates the Sylow $p$-subgroup of $\langle \gamma\rangle$.
It follows from Theorem \ref{one} that $\langle X,Y\rangle_p\langle Z^k\rangle$ is a Sylow $p$-subgroup of $W$.
Let $\eta: \langle X,Y\rangle\to \langle X,Y\rangle_p$ be the canonical projection, and set $x=X^\eta$ and $y=Y^\eta$.
Then
\begin{equation}
\label{xzk}
x^{Z^k}=x^{\alpha^k},\; {}^{Z^k} y=y^{\beta^k}. 
\end{equation}
As $Z^\gamma=[X,Y]$, we have $Z^\gamma=[x,y]w$, where $w$ belongs the the product of the remaining Sylow subgroups of 
$\langle X,Y\rangle$. Let $q=|G/G_p|$, which is relatively prime to $p$, so
\begin{equation}
\label{xzk2}
Z^{q\gamma}=[x,y]^q.
\end{equation}

On the other hand, we have the canonical projection $\pi:G\to G_p$, and we set $a=A^\pi$, $b=B^\pi$, and $c=[a,b]$.
By \cite[Corollary 5.2]{MS}, $G_p$ is generated by $a,b$ subject to the defining relations
\begin{equation}
\label{defrelgp}
a^{[a,b]}=a^{\alpha^\gamma},\; b^{[b,a]}=b^{\beta^\gamma},\; a^u=1,\; b^v=1,
\end{equation}
where $u$ and $v$ are the orders of $a$ and $b$, respectively. The relation $[a^{\alpha^\gamma},b^{\mu^\gamma}]=[a,b]$ is valid in $G_p$
and is inherited from $G$. It follows from Theorem \ref{one} that $N_p$ is the normal closure in $G_p$ 
of the following $\gamma-1$ elements:
\begin{equation}
\label{robin}
[a^\alpha,b^\mu][a,b]^{-1},\dots, [a^{\alpha^{\gamma-1}},b^{\mu^{\gamma-1}}][a,b]^{-1}.
\end{equation}
Thus, setting $a_0=a N_p\in G_p/N_p$ and $b_0=b N_p\in G_p/N_p$, we see that $G_p/N_p$
is generated by $a_0$ and~$b_0$, subject to the defining relations
\begin{equation}
\label{defrelgpnp}
a_0^{[a_0,b_0]}=a_0^{\alpha^\gamma},\; 
b_0^{[b_0,a_0]}=b_0^{\beta^\gamma},\; a_0^u=1,\; b_0^v=1,
[a_0^\alpha,b_0^\mu]=[a_0,b_0],\dots, 
[a_0^{\alpha^{\gamma-1}},b_0^{\mu^{\gamma-1}}]=[a_0,b_0].
\end{equation}
Moreover, the relation $[a_0^{\alpha^\gamma},b_0^{\mu^\gamma}]=[a_0,b_0]$ is valid in $G_p/N_p$. 

The defining relations of $G_p/N_p$ together with 
$[a_0^{\alpha^\gamma},b_0^{\mu^\gamma}]=[a_0,b_0]$ show that the assignment
\begin{equation}
\label{assp}
a_0\mapsto a_0^\alpha, b_0\mapsto b_0^\mu
\end{equation}
preserves all defining relations of $G_p/N_p$. Thus (\ref{assp}) 
extends to an endomorphism, say $\xi$, of $G_p/N_p$. It is an automorphism since $\gcd(\alpha,o(a_0))=1=\gcd(\beta,o(b_0))$.
Thus $\chi=\xi^{q k}$ is also an automorphism of $G_p/N_p$. Here $\chi^{p^n}$ is
conjugation by $[a_0,b_0]^q$ by (\ref{defrelgpnp})  and (\ref{assp}). Moreover, $\chi$ fixes $[a_0,b_0]$ and hence $[a_0,b_0]^q$, because
$
[a_0^{\alpha^{qk}},b_0^{\mu^{qk}}]=[a_0,b_0]
$
follows by applying $qk$ times the automorphism~$\xi$ to 
$
[a_0^\alpha,b_0^\mu]=[a_0,b_0].
$
Thus, by Theorem \ref{z}, there is a group $E=\langle a_0,b_0,d_0\rangle$ such that $\langle a_0,b_0\rangle$ is normal in $E$; $E/\langle a_0,b_0\rangle$
is cyclic of order $p^n$; $d_0$ has order $p^n$ modulo 
$\langle a_0,b_0\rangle$; and 
\begin{equation}
\label{rele}
a_0^{d_0}=a_0^{\alpha^{qk}},\; b_0^{d_0}=b_0^{\mu^{qk}},\; d_0^{p^n}=[a_0,b_0]^q. 
\end{equation}

The homomorphism (\ref{omega}) yields an isomorphism
$
\Delta:G_p/N_p\to \langle X,Y\rangle_p,
$
given by
$
a_0\mapsto x,\; b_0\mapsto y.
$
It follows from (\ref{xzk}), (\ref{xzk2}), and (\ref{rele}) that $\Delta$
extends to a homomorphism $E\to W_p$ via $d_0\mapsto Z^{qk}$. As $|E|=|G_p/N_p|p^n=|\langle X,Y\rangle_p|p^n=|W_p|$,
we see that $\Delta$ is an isomorphism. The given presentation of $G_p/N_p$ can now be easily
adapted to yield a presentation of $E\cong W_p$. 

\begin{theorem}\label{dos} Let $p\in\N$ be a prime factor of $(\alpha^\gamma-1)(\beta^\gamma-1)$, where
$\gamma=p^n k$, $n\geq 0$, $p\nmid k$. Then the Sylow $p$-subgroup $G_p$ of $G$ is generated
by elements $a$ and $b$, subject to the defining relations~(\ref{defrelgp}), where
$u$ and $v$ are the orders of $a$ and  $b$ in $G_p$. Moreover, $N_p$ is the normal closure
in $G_p$ of the $\gamma-1$ elements (\ref{robin}); $G_p/N_p$ is generated
by elements $a_0$ and~$b_0$, subject to the defining relations~(\ref{defrelgpnp}), 
where $\mu$ is inverse of $\beta$ modulo the order of $b_0$; and $G_p/N_p$ imbeds as a normal subgroup of 
a Sylow $p$-subgroup $W_p$ of $W$.
Furthermore, $W_p$ is generated by elements 
$a_0,b_0,d_0$ subject to defining relations~(\ref{defrelgpnp}) and (\ref{rele}), where 
$q$ is the order of $G/G_p$,  and $W/\langle a_0,b_0\rangle$ cyclic 
of order $p^n$. Thus
$$
|W_p|=p^n |G_p|/|N_p|.
$$
In particular, if $p\nmid \gamma$, then $G_p/N_p$ is isomorphic to the sole Sylow $p$-subgroup of $W$.
\end{theorem}

\begin{theorem}\label{tres} Let $\mu$ be an integer that is inverse of $\beta$ modulo the order of $b$ and write $\gamma=p^n k$,
where $n\geq 0$ and $p\nmid k$. Suppose that $p|(\alpha-1)$ and $p|(\beta-1)$. Let $M_p$ be
 the normal closure in $G_p$ of the $p^n-1$ elements
$
[a^\alpha,b^\mu][a,b]^{-1},\dots, [a^{\alpha^{p^n-1}},b^{\mu^{p^n-1}}][a,b]^{-1}.
$
Then $M_p=N_p$. In particular, if  $p\nmid \gamma$, then $G_p$ imbeds as the sole Sylow $p$-subgroup of $W$.
\end{theorem}

\begin{proof} Let $u$ and $v$ be the orders of $a$ and $b$, respectively, so that
$u=p^{u_0}$ and $v=p^{v_0}$ for some $u_0,v_0\geq 0$ (in fact, $u_0,v_0\geq 1$, as indicated in
\cite[pp. 602-603]{M}). Let $r$ and $s$ be the orders of $\alpha$ and $\beta$ modulo $u$ and $v$,
respectively. As $\alpha\equiv 1\equiv \beta\mod p$, it follows that $r=p^{r_0}$ and $s=p^{s_0}$
for some $r_0,s_0\geq 0$. Let $t=\max\{r_0,s_0\}$. Since $p\nmid k$, there is $i\in\Z$ such that
$ik\equiv 1\mod p^t$, so
$$
\gamma i\equiv p^n k i\equiv p^n\mod r,\; \gamma i\equiv p^n k i\equiv p^n\mod s.
$$
The definitions of $r$ and $s$ then yield
$$
\alpha^{\gamma i}=\alpha^{p^n}\mod u,\; \beta^{\gamma i}=\beta^{p^n}\mod v.
$$
As $\beta\mu\equiv 1\mod v$, we also have
$$
\mu^{\gamma i}=\mu^{p^n}\mod v.
$$
Thus
$$
a^{c_i}=a^{\alpha^{\gamma i}}=a^{\alpha^{p^n}},\; b^{c_i}=b^{\mu^{\gamma i}}=b^{\mu^{p^n}},
$$
and therefore
\begin{equation}\label{abpn}
[a,b]=[a,b]^{c_i}=[a^{c_i},b^{c_i}]=[a^{\alpha^{p^n}},b^{\mu^{p^n}}].
\end{equation}
A presentation for $G_p$ is given in Theorem \ref{dos}. We may now mimic the proof of Theorem \ref{dos},
with (\ref{abpn}) instead of $[a^{\alpha^{\gamma}},b^{\mu^{\gamma}}]$, to conclude that $M_p=N_p$.
\end{proof}

\section{The case when $\alpha=\beta$}\label{s3}

We keep the notation and hypotheses from Section \ref{s2}, and assume from now on that $\alpha=\beta$.
We fix $p\in S$, let $m=v_p(\alpha^\gamma-1)$, and
set $J=G_p$. Note that $\mu$ is inverse of $\alpha$ modulo the order of~$b$, and hence modulo
the order of $c=[a,b]$ by \cite[pp. 609]{M} or \cite[Proposition 7.2]{MS}. 

\begin{prop}\label{adentro} We have $N_p\subseteq Z_3(J)$ in Cases 1 and 2, 
and $N_p\subseteq Z_5(J)$ in Case 3.
\end{prop}

\begin{proof} It follows from \cite[Theorem 8.1 and Proposition 8.2]{MS} that $c$ is in the center of $J$ modulo $Z_3(J)$
in Cases 1 and~2, and modulo $Z_5(J)$ in Case 3. Thus, for all integers $i,j$, we have
$$
[a^i,b^j]\equiv [a,b]^{ij}\mod Z_3(J)\text{ in Cases 1 and 2, and } [a^i,b^j]\equiv [a,b]^{ij}\mod Z_5(J)\text{ in Case 3}.
$$
Since $\alpha\mu$ is congruent to 1 modulo the order of $c$, all $[a^{\alpha^i},b^{\mu^i}][a,b]^{-1}$, with $i\geq 1$, are in 
the normal subgroup $Z_3(J)$ of $J$ in Cases 1 and 2, and in the normal subgroup $Z_5(J)$ of $J$ in Case~3. This proves
the stated inclusions by Theorem \ref{one}. 
\end{proof}


By assumption $p\in S$, so $m=v_p(\alpha^\gamma-1)>0$. We henceforth fix the following notation:
$$
\gamma=p^n k,\; \alpha=1+p^h t,\; \alpha^\gamma=1+p^m\ell,\; q=|G/G_p|,
$$
where $n\geq 0$, $p\nmid k$, $h\geq 0$, $p\nmid t$,  and $p\nmid \ell$. As seen in Section \ref{s2},
$N_p$ is the normal closure in $J$ of 
\begin{equation}
\label{defvi}
V_i=[a^{\alpha^i},b^{\mu^i}][a,b]^{-1},\quad i\geq 1.
\end{equation}
The following elements of $J$ play an important role in the determination of $N_p$ in Cases 1 and 2:
\begin{equation}
\label{deffi}
F_i=a^{-p^{e(m)}(\alpha^i-1)} b^{p^{e(m)}(\mu^i-1)}=a^{-p^{e(m)}(\alpha-1)u_i} b^{p^{e(m)}(\mu-1)v_i},\quad i\geq 1,
\end{equation}
where $e(m)=m$ if $p$ is odd and $e(m)=m-1$ if $p=2$, and
\begin{equation}
\label{defui}
u_i=1+\alpha+\cdots+\alpha^{i-1},\; v_i=1+\mu+\cdots+\mu^{i-1},\; w_i=v_i-u_i,\quad i\geq 1,
\end{equation}
noting that
$$
u_1=1=v_1,\; w_1=0,\; u_2=1+\alpha,\; v_2=1+\mu,\; w_2=\mu-\alpha.
$$
By \cite[Propositions 9.1 and 9.2]{MS}, $Z_3(J)$ is abelian in Case 1 as well as in Case 2 when $m>1$.
Thus, in both cases, we have
$$
F_i=a^{-p^{e(m)}(\alpha-1)u_i} b^{p^{e(m)}(\mu-1)u_i}b^{p^{e(m)}(\mu-1)w_i}=
a^{-p^{e(m)}(\alpha-1)v_i} b^{p^{e(m)}(\mu-1)v_i}a^{p^{e(m)}(\alpha-1)w_i},\quad i\geq 1,
$$
and therefore
\begin{equation}
\label{vui}
F_1^{-v_i} F_i=a^{p^{e(m)}(\alpha-1)w_i},\; F_1^{-u_i} F_i=b^{p^{e(m)}(\mu-1)w_i},
\end{equation}
whence
\begin{equation}
\label{vui2}
F_i\in \langle F_1, a^{p^{e(m)}(\alpha-1)w_i}\rangle,\; F_i\in \langle F_1, b^{p^{e(m)}(\alpha-1)w_i}\rangle,\quad i\geq 1.
\end{equation}

In the calculation of (\ref{defvi}), we will also make use of
\[
\phi(i) = \frac{(i-1)i}{2},
\qquad\varphi(i) = \frac{i(i-1)(i-2)}{6},\quad i\in\Z.
\]

\section{Case 1 when $\alpha=\beta$}\label{s4}

We keep the notation and hypotheses from Sections \ref{s2} and \ref{s3}, and assume further 
that $p$ is odd and that if $p=3$, then $\alpha^\gamma\not\equiv 7\mod 9$ or,
alternatively, that if $(p,m)=(3,1)$, then $(\alpha^\gamma-1)/3\not\equiv -1\mod 3$.
As $p$ is odd,  the map $x\mapsto x^2$ is an automorphism of any abelian subgroup of $J$,
with inverse automorphism $x\mapsto x^{1/2}$ of square  $x\mapsto x^{1/4}$.
In this section we describe $N_p, G_p/N_p$, and $W_p$.


In relation to (\ref{defrelgp}), \cite[Theorem 7.1 and Proposition 7.2]{MS} ensure that $o(a)=o(b)=p^{3m}$,
$o(c)=p^{2m}$, and $|J|=p^{7m}$. Thus $J$ is generated by  $a$, $b$, and $c$ subject to the defining relations
\begin{equation}
\label{relgp}
c=[a,b], a^{c}=a^{\alpha^\gamma},\, {}^{c} b=b^{\alpha^\gamma}, a^{p^{3m}}=1, b^{p^{3m}}=1.
\end{equation}
Notice the existence of the automorphism $a\leftrightarrow b$, $c\leftrightarrow c^{-1}$ of $J$.

\begin{theorem}\label{thm.commutators.J}
    For all $i,j\in\Z$ the following commutator formulas hold in $J$:
    \begin{align}
        &[c^i,a^j]
        = a^{-p^m\ell ij - p^{2m}\ell^2\phi(i)j},\label{eq.comj.2}\\
        &[c^i,b^j]
        = b^{p^m\ell ij - p^{2m}\ell^2\phi(i+1)j},\label{eq.comj.3}\\
        &[a^i,b^j]
        = a^{-p^m\ell\phi(i)j} b^{p^m\ell i\phi(j)} c^{ij - p^m\ell\phi(i)\phi(j)} a^{\xi(i,j)},\label{eq.comj.4}
    \end{align}
    where 
    \[
    \xi(i,j)
    = (p^{2m}\ell^2/2)\{2\varphi(i+1)j + (2i-7)\phi(i)\phi(j) - 2i\phi(j) - (3i+1)i\varphi(j)-2\delta_{p,3}3^{m-1}\ell\phi(i)\phi(j)\},
    \]
		with $\delta_{p,3}=1$ if $p=3$ and $\delta_{p,3}=0$ if $p\neq 3$.
		\end{theorem}

This is shown in \cite[Theorem 8.1]{MS2}. A proof independent of \cite{MS2} is given in Section \ref{andrea}.


By \cite[Proposition 6.4]{MS} the following fundamental relation holds in $J$:
\begin{equation}
\label{funrel}
a^{p^{2m}}b^{p^{2m}}=1.
\end{equation}
Moreover, by \cite[Theorem 8.1]{MS}, 
\begin{equation}
\label{zj}
Z(J)=\langle a^{p^{2m}}\rangle=\langle b^{p^{2m}}\rangle, 
Z_2(J)=\langle a^{p^{2m}}, c^{p^m}\rangle, Z_3(J)=\langle a^{p^m}, b^{p^m}, c^{p^m}\rangle,
\end{equation}
where $Z_3(J)$ is abelian by \cite[Proposition 9.1]{MS}. Furthermore, by \cite[Proposition 7.2]{MS},
\begin{equation}
\label{inter1}
\langle a,c\rangle\cap \langle b\rangle=Z(J)=\langle b,c\rangle\cap \langle a\rangle,
\end{equation}
\begin{equation}
\label{inter2}
\langle a\rangle\cap \langle c\rangle=1=\langle b\rangle\cap \langle c\rangle.
\end{equation}

Since $\alpha\equiv 1\mod p^h$, we infer $\alpha^i\equiv 1\mod p^h$, and as $\alpha^i\mu^i\equiv 1\mod p^{3m}$,
we can multiply both sides of $\alpha^i\equiv 1\mod p^h$ by $\mu^i$ to derive $1\equiv \mu^i\mod p^h$, $i\geq 1$.
Suppose, if possible, that $\mu\equiv 1\mod p^{h+1}$. Then multiplication by $\alpha$
yields $1\equiv \alpha\mod p^{h+1}$, a contradiction. Thus
\begin{equation}
\label{vpm}
v_p(\alpha-1)=h=v_p(\mu-1).
\end{equation}

Making use of (\ref{defvi}), (\ref{zj}), Theorem \ref{thm.commutators.J}, and $\alpha^i\mu^i\equiv 1\mod p^{3m}$, we see that
\begin{equation}
\label{vi}
V_i=a^{-(p^m\ell/2)(\alpha^i-1)} b^{(p^m\ell/2)(\mu^i-1)} c^{-(p^m\ell/4)(\alpha^i-1)(\mu^i-1)} z_i, \quad i\geq 1,
\end{equation}
where $z_i\in Z(J)$. 

\medskip

\noindent{\bf The case $h>0$.} We assume that $h>0$ throughout this subsection. Then $m=h+n$.

Observe that Theorem \ref{thm.commutators.J} yields $z_i^{p^n}=1$ in (\ref{vi}), because
$p^h$ is a factor of $\phi(\alpha^i)$, $\phi(\mu^i)$, $\varphi(\mu^i)$, and $\varphi(\alpha^i+1)$,
and we have $m=h+n$. Since $Z_3(J)$ is abelian, Proposition \ref{adentro} together with
$\alpha^i\equiv 1\equiv\mu^i\mod p^h$, $c^{p^{2m}}=1$, $m=h+n$, 
$z_i^{p^n}=1$, (\ref{zj}), and (\ref{vi}) yield
\begin{equation}
\label{vip}
V_i^{p^n}=a^{-(p^{m+n}\ell/2)(\alpha^i-1)} b^{(p^{m+n}\ell/2)(\mu^i-1)}\in Z(J),\quad i\geq 1.
\end{equation}
Let $M_{n}$ be the subgroup of $Z_3(J)$ of all
$V$ such that $V^{p^{n}}\in Z(J)$. This is a characteristic subgroup of $Z_3(J)$. As 
$Z_3(J)$ is a normal subgroup of $J$, it follows that $M_n$ is a normal subgroup of $J$.
As all $V_i\in M_n$, we infer that $N_p$ is contained in $M_{n}$. Note that (\ref{funrel}) and (\ref{vip}) give
\begin{equation}
\label{vip2}
V_i^{p^n}=a^{-(p^{m+n}\ell/2)[(\alpha^i-1)+(\mu^i-1)]},\quad i\geq 1.
\end{equation}
Set
$$
\delta_i=-(i p^h t+ \binom{i}{2} p^{2h} t^2+\cdots+ \binom{i}{i} p^{ih} t^i),\quad i\geq 1,
$$
so that
$
\alpha^i=1-\delta_i.
$
As $\alpha^i\mu^i\equiv 1\mod p^{3m}$, we infer
$
\mu^i\equiv 1+\delta_i+\delta_i^2+\cdots+\delta_i^{p^{3m}-1}\mod p^{3m}.
$
Thus
\begin{equation}
\label{cru}
(\alpha^i-1)+(\mu^i-1)\equiv \delta_i^2+\cdots+\delta_i^{p^{3m}-1}\mod p^{3m},
\end{equation}
and therefore
\begin{equation}
\label{vip3}
v_p((\alpha^i-1)+(\mu^i-1))\geq 2h+2v_p(i),\quad i\geq 1.
\end{equation}
It follows from (\ref{vip2}) and (\ref{vip3}) that
$V_i^{p^{2n}}=1$ for all $i\geq 1$.

Let $L_{2n}=\Omega_{2n}(Z_3(J))$ be the subgroup of $Z_3(J)$ of all
$V$ such that $V^{p^{2n}}=1$. This is a characteristic subgroup of $Z_3(J)$. As 
$Z_3(J)$ is a normal subgroup of $J$, it follows that $L_{2n}$ is a normal subgroup of $J$.
As all $V_i\in L_{2n}$, we infer that $N_p$ is contained in $L_{2n}$.

The right hand side of (\ref{cru}) contains the term $i^2 p^{2h} t^2$ while
all other terms are multiples of $p^{3h}$. Thus, if $p\nmid i$, then $p\nmid i^2 t^2$, whence
\begin{equation}
\label{val}
v_p((\alpha^i-1)+(\mu^i-1))=2h,\quad p\nmid i. 
\end{equation}

\begin{prop}\label{v2n} If $p\nmid i$, then $V_i$ has order $p^{2n}$.
\end{prop}

\begin{proof} Immediate consequence of (\ref{vip2}) and (\ref{val}).
\end{proof}

\begin{prop}\label{nuevoandrea} We have
$$
Z_3(J)=\langle a^{p^{m}}, a^{p^{m}}b^{p^{m}},c^{p^{m}}\rangle=\langle b^{p^{m}},a^{p^{m}}b^{p^{m}},c^{p^{m}}\rangle
\cong C_{p^{2m}}\times  C_{p^{m}}\times C_{p^{m}},
$$
where $Z_3(J)$ is the internal direct product of $\langle a^{p^{m}}\rangle$, $\langle a^{p^{m}}b^{p^{m}}\rangle$,
and $\langle c^{p^{m}}\rangle$.
\end{prop}

\begin{proof} Use (\ref{funrel})-(\ref{inter2}), that $Z_3(J)$ is abelian, and the stated orders of $a$, $b$, and $c$.
\end{proof}

\begin{prop}\label{nuevoandrea2} We have
$$
M_n=\langle a^{p^{m+h}}, a^{p^{m+h}}b^{p^{m+h}},c^{p^{m+h}}\rangle=\langle b^{p^{m+h}}, a^{p^{m+h}}b^{p^{m+h}},c^{p^{m+h}}\rangle.
$$
\end{prop}

\begin{proof} Immediate consequence of Proposition \ref{nuevoandrea} and the definition of $M_n$.
\end{proof}

\begin{prop}\label{nuevoandrea3} We have
$$
L_{2n}=\begin{cases} \langle a^{p^{m+2h}}, a^{p^{m}}b^{p^{m}},c^{p^{m}}\rangle=\langle b^{p^{m+2h}}, a^{p^{m}}b^{p^{m}},c^{p^{m}}
\rangle\text{ if }n\geq h,\\
\langle a^{p^{m+2h}}, a^{p^{2h}}b^{p^{2h}},c^{p^{2h}}\rangle=\langle b^{p^{m+2h}}, a^{p^{2h}}b^{p^{2h}},c^{p^{2h}}
\rangle\text{ if }n\leq h.
\end{cases}
$$
\end{prop}

\begin{proof} Immediate consequence of Proposition \ref{nuevoandrea} and the definition of $L_{2n}$.
\end{proof}

\begin{prop}\label{mln} We have
$$
M_n\cap L_{2n}=\langle a^{p^{m+2h}},a^{p^{m+h}}b^{p^{m+h}},c^{p^{m+h}}\rangle=
\langle b^{p^{m+2h}},a^{p^{m+h}}b^{p^{m+h}},c^{p^{m+h}}\rangle
\cong C_{p^{2n}}\times  C_{p^{n}}\times C_{p^{n}},
$$
where $a^{p^{m+2h}}$ and $b^{p^{m+2h}}$ have order $p^{2n}$, $a^{p^{m+h}}b^{p^{m+h}}$ has order $p^n$, 
and $c^{p^{m+h}}$ has order $p^n$.
\end{prop}

\begin{proof} Immediate consequence of Propositions \ref{nuevoandrea2} and \ref{nuevoandrea3}.
\end{proof}

\begin{prop}\label{ordnz} We have $\Omega_n(Z(J))=\langle a^{p^{2m+h}}\rangle\subseteq N_p$.
\end{prop}

\begin{proof} We know from \eqref{vip} that $V_1^{p^n}\in Z(J)$. By Proposition \ref{v2n},  the order of $V_1^{p^n}$ is~$p^n$.
Thus $\langle V_1^{p^n}\rangle$ is a subgroup of $Z(J)$ of order $p^n$.
Being cyclic, $Z(J)$ has only one subgroup of order $p^n$, namely $\langle a^{p^{2m+h}}\rangle$.
Thus $N_p\supseteq \langle V_1^{p^n}\rangle=\langle a^{p^{2m+h}}\rangle=\Omega_n(Z(J))$.
\end{proof}

\begin{prop}
\label{nor1} We have $c^{p^{m+h}}\in N_p$.
\end{prop}

\begin{proof} Applying Theorem \ref{thm.commutators.J} to (\ref{vi}) yields
$
V_1^a=V_1 c^{-p^m(\ell/2)(\mu-1)} x,
$
where $x\in Z(J)$ and $x^{p^n}=1$. It now follows from Proposition \ref{ordnz} that $c^{p^{m+h}}\in N_p$.
\end{proof}

We know that all $z_i$ appearing in (\ref{vi}) satisfy $z_i\in Z(J)$ and $z_i^{p^n}=1$. Thus, all $z_i\in N_p$ by Proposition  \ref{ordnz}.
All $c^{-p^m(\ell/4)(\alpha^i-1)(\mu^i-1)}$ appearing
in (\ref{vi}) belong to $\langle c^{p^{m+h}}\rangle$, which is contained in $N_p$ by Proposition~\ref{nor1}. It follows 
that all $F_i$ appearing in (\ref{deffi}) are in $N_p$.

Recall that $\alpha=1+p^h t$, which implies $\mu=1-p^h t+p^{2h} f$, where $p\nmid f$. Thus $\mu=1+p^h t_0$,
where $t_0=-t+p^h f$ is not a multiple of $p$, and  $\alpha=1-p^h t_0+p^{2h}f_0$, where $p\nmid f_0$. 

Using (\ref{vpm}) and $v_p(\mu-\alpha)=h$,
we deduce from (\ref{vui}) with $i=2$ that
\begin{equation}
\label{enn2}
a^{p^{m+2h}}\in N_p,\; b^{p^{m+2h}}\in N_p,
\end{equation}
and from (\ref{vui2}) that
\begin{equation}
\label{enn22}
F_i\in \langle F_1, a^{p^{m+2h}}\rangle,\; F_i\in \langle F_1, b^{p^{m+2h}}\rangle,\quad i\geq 1.
\end{equation}
On the other hand, taking $i=1$ in (\ref{deffi}) yields
$$
F_1=a^{-p^{m+h}t} b^{-p^{m+h}t}b^{p^{m+2h}f}=a^{p^{m+h}t_0} b^{p^{m+h}t_0}a^{p^{m+2h}f_0},
$$
so (\ref{enn2}) and $F_1\in N_p$ imply
\begin{equation}
\label{enn23}
a^{p^{m+h}} b^{p^{m+h}}\in N_p.
\end{equation}
It now follows from Propositions \ref{mln} and \ref{nor1}, as well as (\ref{enn2}) and (\ref{enn23}), that 
$M_n\cap L_{2n}\subseteq N_p$. But we already know that $N_p \subseteq M_n\cap L_{2n}$, so $N_p=M_n\cap L_{2n}$.

\begin{theorem}\label{main1} We have
$$
N_p=M_n\cap L_{2n}=\langle a^{p^{m+2h}}, a^{p^{m+h}}b^{p^{m+h}},c^{p^{m+h}}\rangle=
\langle b^{p^{m+2h}}, a^{p^{m+h}}b^{p^{m+h}},c^{p^{m+h}}\rangle
\cong C_{p^{2n}}\times  C_{p^{n}}\times C_{p^{n}};
$$
$$
|N_p|=p^{4n},\; |G_p/N_p|=p^{7h+3n},\; |W_p|=p^{7h+4n}=p^{7m-3n};
$$
$G_p/N_p$ is generated by elements $a_0,b_0,c_0$ subject to defining relations
\begin{equation}
\label{rel0}
c_0=[a_0,b_0], a_0^{c_0}=a_0^{\alpha^\gamma},\, {}^{c_0} b_0=b_0^{\alpha^\gamma},
\end{equation}
\begin{equation}
\label{rel1}
a_0^{p^{m+2h}}=1, a_0^{p^{m+h}}b_0^{p^{m+h}}=1, b_0^{p^{m+2h}}=1, c_0^{p^{m+h}}=1,
\end{equation}
and every element of $G_p/N_p$ can be written in one and only way in the form
\begin{equation}
\label{write}
a_0^{e_1} b_0^{e_2}c_0^{e_3},\quad 0\leq e_1<p^{m+2h}, 0\leq e_2<p^{m+h}, 0\leq e_3<p^{m+h};
\end{equation}
$W_p$ is generated by elements $a_0,b_0,c_0,d_0$ subject to the defining relations (\ref{rel0}), (\ref{rel1}),
as well as
\begin{equation}
\label{rel2}
a_0^{d_0}=a_0^{\alpha^{qk}},\; {}^{d_0} b_0=b_0^{\alpha^{qk}},\; d_0^{p^n}=c_0^q,
\end{equation}
and every element of $W_p$ can be written in one and only way in the form
$$
a_0^{e_1} b_0^{e_2}c_0^{e_3}d_0^{e_4},\quad 0\leq e_1<p^{m+2h}, 0\leq e_2<p^{m+h}, 0\leq e_3<p^{m+h}, 0\leq e_4<p^{n}.
$$
Moreover, the automorphism $a\leftrightarrow b$, $c\leftrightarrow c^{-1}$ of $G_p$ induces an automorphism $a_0\leftrightarrow b_0$, 
$c_0\leftrightarrow c_0^{-1}$ on $G_p/N_p$, which extends to an automorphism of $W_p$ via $d_0\leftrightarrow d_0^{-1}$.
\end{theorem}

\begin{proof} All statements up to and including $|N_p|=p^{4n}$ have already been proven. It now follows from \cite[Theorem 7.1]{MS} that
$
|G_p/N_p|=p^{7h+7n}/p^{4n}= p^{7h+3n},
$
whence by Theorem \ref{dos}
$
|W_p|=p^{7h+4n}.
$
We see from $N_p=\langle a^{p^{m+h}}b^{p^{m+h}},a^{p^{m+2h}},c^{p^{m+h}}\rangle$ and (\ref{relgp}) that
$G_p/N_p$ has the stated presentation. It follows from \cite[Lemma 6.1 or Theorem 7.1]{MS} that every element
of $G_p/N_p$ can the written as stated in (\ref{write}). Uniqueness follows from the fact that 
$|G_p/N_p|=p^{7h+3n}$. The remaining statements about $W_p$ now follow from Theorem \ref{dos} and $|W_p|=p^{7h+4n}$.
\end{proof}

\medskip

\noindent{\bf The case $h=0$.} We assume that $h=0$ throughout this subsection, that is, $\alpha\not\equiv 1\mod p$.

\begin{prop}\label{h0z2} Set $U=\langle a^{p^{2m}}, c^{p^m}, a^{-p^m(\alpha-1)} b^{p^m(\mu-1)}\rangle$.
Then $U\subseteq N_p$ and $U$ has order $p^{3m}$.
\end{prop}

\begin{proof} By (\ref{vi}), we have
\begin{equation}
\label{generic}
V_1=a^{-(p^m\ell/2) (\alpha-1)} b^{(p^m\ell/2) (\mu-1)}c^{-(p^m \ell/4) (\alpha-1)(\mu-1)}d,
\end{equation}
where $\alpha\not\equiv 1\mod p$, $\mu\not\equiv 1\mod p$, and $d\in Z(J)$. It follows from Theorem \ref{thm.commutators.J} that
$$
[V_1,a]=c^{p^m f}e\in N_p,
$$
where $f\in\Z$ is not a multiple of $p$ and $e\in Z(J)$. A second use of Theorem \ref{thm.commutators.J} reveals that
$$
[V_1,a,a]=a^{p^{2m}s}\in N_p,
$$
where $s\in\Z$ is not a multiple of $p$. It follows that $Z_2(J)=\langle a^{p^{2m}}, c^{p^m}\rangle\subseteq N_p$.
Going back to (\ref{generic}), we infer that $a^{-p^m(\alpha-1)} b^{p^m(\mu-1)}$ is in $N_p$ and has order $p^m$
modulo $Z_2(J)$.
\end{proof}

\begin{cor}\label{clh0} $N_p$ is the normal closure of $\{a^{p^{2m}}, c^{p^m}\}\cup\{a^{-p^m(\alpha^i-1)} b^{p^m (\mu^i-1)}\,|\, i\geq 1\}$
in $J$.
\end{cor}

\begin{theorem}\label{hzero} Suppose that $\alpha\not\equiv \pm 1\mod p$. Then $N_p=Z_3(J)$ and
$G_p/N_p\cong\mathrm{Heis}(\Z/p^m\Z)$, the Heisenberg group over $\Z/p^m\Z$, so  $G_p/N_p$
has order $p^{3m}$ and is generated by elements
$a_0$, $b_0$, and $c_0$ subject to the defining
relations
\begin{equation}\label{reze}
c_0=[a_0,b_0], a_0^{c_0}=a_0, b_0^{c_0}=b_0, a_0^{p^{m}}=b_0^{p^{m}}=c_0^{p^{m}}=1,
\end{equation}
and $W_p$ is a cyclic extension of $G_p/N_p$ of order $p^{3m+n}$ generated by elements $a_0,b_0,c_0,d_0$
with defining relations (\ref{rel2}) and (\ref{reze}).  Moreover, every element of $W_p$ can be written in one and only way in the form
$$
a_0^{e_1} b_0^{e_2}c_0^{e_3}d_0^{e_4},\quad 0\leq e_1,e_2,e_3<p^{m}, 0\leq e_4<p^{n}.
$$
Furthermore, the automorphism $a\leftrightarrow b$, $c\leftrightarrow c^{-1}$ of $G_p$ induces an automorphism $a_0\leftrightarrow b_0$, 
$c_0\leftrightarrow c_0^{-1}$ on $G_p/N_p$, which extends to an automorphism of $W_p$ via $d_0\leftrightarrow d_0^{-1}$.

\end{theorem}

\begin{proof} We know from Propositions \ref{adentro} and \ref{h0z2} that $Z_2(J)\subseteq N_p\subseteq Z_3(J)$.
Here $Z_3(J)/Z_2(J)$ is a free module of rank 2 over $\Z/\Z_{p^m}$ with basis $\{a^{p^m}Z_2(J), b^{p^m}Z_2(J)\}$.
Applying (\ref{vi}) with $i=1$ and $i=2$ and appealing to Proposition \ref{h0z2}, we see that
$$
a^{-p^m(\alpha-1)} b^{p^m(\mu-1)}\in N_p,\; a^{-p^m(\alpha^2-1)} b^{p^m(\mu^2-1)}\in N_p.
$$
The determinant of the matrix
$$
\left(\begin{matrix} -(\alpha-1) & \mu-1 \\ -(\alpha^2-1) & \mu^2-1
\end{matrix}\right)
$$
is $(\alpha-1)(\mu-1)(\alpha-\mu)$, which is relatively prime to $p$, as $\alpha\equiv\mu\mod p$ leads to
the contradiction $\alpha^2\equiv 1\mod p$. It follows that 
$$
\{a^{-p^m(\alpha-1)}b^{p^m(\mu-1)}Z_2(J), a^{-p^m(\alpha^2-1)} b^{p^m(\mu^2-1)} Z_2(J)\}
$$
is a $\Z/\Z_{p^m}$-basis of $Z_3(J)/Z_2(J)$, which forces $N_p=Z_3(J)$. That $G_p/Z_3(J)\cong\mathrm{Heis}(\Z/p^m\Z)$ 
was proven in \cite[Proposition 9.1]{MS}.
\end{proof}


 
Assume that $\alpha\equiv -1\mod p$, say $\alpha=-1+p^{h_0} t_0$,
where $h_0>0$ and $p\nmid t_0$. As $\alpha^\gamma\equiv 1\mod p^m$, we see that $\gamma=2\gamma_0$
for some $\gamma_0>0$, so
$$
m=v_p(\alpha^\gamma-1)=v_p(\alpha^{2\gamma_0}-1)=h_0+n,
$$
given that $v_p(\alpha^2-1)=h_0$ and $v_p(\gamma_0)=n$. 

\begin{theorem}\label{hzero2} Suppose that $\alpha\equiv -1\mod p$. Then
$$
N_p=\langle c^{p^m}, a^{-p^m(\alpha-1)} b^{p^m(\mu-1)}, a^{p^{m+h_0}}\rangle=
\langle c^{p^m}, a^{-p^m(\alpha-1)} b^{p^m(\mu-1)}, b^{p^{m+h_0}}\rangle
$$
has order $p^{4m-h_0}$, while $|G_p/N_p|=p^{3m+h_0}$ and $|W_p|=p^{4m}$. Moreover,
$G_p/N_p$ is generated by elements $a_0,b_0,c_0$ subject to defining relations (\ref{rel0}) and
\begin{equation}
\label{rel10}
a_0^{p^{m+h_0}}=1, a_0^{p^{m}(\alpha-1)}=b_0^{p^{m}(\mu-1)}, b_0^{p^{m+h_0}}=1, c_0^{p^{m}}=1,
\end{equation}
and every element of $G_p/N_p$ can be written in one and only way in the form
$$
a_0^{e_1} b_0^{e_2}c_0^{e_3},\quad 0\leq e_1<p^{m+h_0}, 0\leq e_2<p^{m}, 0\leq e_3<p^{m};
$$
$W_p$ is generated by elements $a_0,b_0,c_0,d_0$ subject to the defining relations (\ref{rel0}), (\ref{rel2}),
and (\ref{rel10}), and every element of $W_p$ can be written in one and only way in the form
$$
a_0^{e_1} b_0^{e_2}c_0^{e_3}d_0^{e_4},\quad 0\leq e_1<p^{m+h_0}, 0\leq e_2<p^{m}, 0\leq e_3<p^{m}, 0\leq e_4<p^{n}.
$$
Moreover, the automorphism $a\leftrightarrow b$, $c\leftrightarrow c^{-1}$ of $G_p$ induces an automorphism $a_0\leftrightarrow b_0$, 
$c_0\leftrightarrow c_0^{-1}$ on $G_p/N_p$, which extends to an automorphism of $W_p$ via $d_0\leftrightarrow d_0^{-1}$.
\end{theorem}

\begin{proof} Note first of all that $\langle c^{p^m}, a^{-p^m(\alpha-1)} b^{p^m(\mu-1)}, a^{p^{m+h_0}}\rangle$
is a normal subgroup of $J$, being an intermediate subgroup between $Z_2(J)$ and $Z_3(J)$. 

From $\alpha=-1+p^{h_0} t_0=- (1-p^{h_0} t_0)$, we deduce $\mu=-(1+p^{h_0}t_0+p^{2h_0} u)$ for some $u\in\Z$ relatively
prime to $p$. Using $v_p(\alpha-1)=0=v_p(\mu-1)$ and $v_p(\mu-\alpha)=h_0$, we see from (\ref{vui}) that
$a^{p^{m+h_0}},b^{p^{m+h_0}}\in N_p$. As $w_2=\mu-\alpha$ is a factor of every $w_i$, we deduce from (\ref{vui2})
and Corollary~\ref{clh0} that $\langle c^{p^m}, a^{-p^m(\alpha-1)} b^{p^m(\mu-1)}, a^{p^{m+h_0}}\rangle=N_p=
\langle c^{p^m}, a^{-p^m(\alpha-1)} b^{p^m(\mu-1)}, b^{p^{m+h_0}}\rangle$. 
The orders of $N_p$, $G_p/N_p$, and $W_p$ are easily seen to be correct.
\end{proof}

\section{Witt's Formula}\label{Witt}

According to \cite[Theorem 5.1]{MKS} the rank of the free $\mathbb Z$-module $\gamma_w(F)/\gamma_{w+1}(F)$ for the free
group $F=F\{x_1,\dots,x_r\}$ of rank $r$ is given by 
$$
\frac{1}{w}\sum_{d|w}\mu(\frac{w}{d})r^d,
$$
where $\mu$ denotes the M\"obius function. More explicitly the previous sum counts the number of basic commutators in $F$ of weight $w$ (see {\cite[Ch. 11]{H}).
A more refined version provides the number of such commutators where each free generator $x_i$ occurs $w_i$ times. The result is
$$
\frac{1}{w}\sum_{d|w_i}\mu(d)\frac{(\frac{w}{d})!}{\prod_{i=1}^{r}(\frac{w_i}{d})!},
$$
where $w=\sum_{i=1}^{r}w_i$, $w_i\ge 0$ and $d$ runs over all common divisors of $w_1,\ldots,w_r$.

In particular we are interested in the case $r=2$ and $1\le w\le 5$, since apart from Case $3$ the nilpotency class 
of the Sylow $p$-subgroup $J=G_p$ of the Macdonald group $G(\alpha^\gamma,\alpha^\gamma)$ is $5$.

Applying Witt's Formula we obtain that the ranks of each section are:
$$
2,1,2,3,6.
$$

For $F=F\{x,y\}$, we list here the basic commutators in $x,y$ keeping the ordering exploited in Magma (this ordering is quite arbitrary)
$$
\begin{array}{|r|l|}
  \hline
  \text{Weight}&\text{Basic Commutators}\\
  \hline
  1 & x,y \\
  2 & [y,x] \\
  3 & [y,x,x], [y,x,y] \\
  4 & [y,x,x,x],[y,x,y,x], [y,x,y,y] \\
  5 & [y,x,x,x,x], [y,x,x,x,y],[y,x,y,x,x],[y,x,x,y,y],[y,x,y,y,x], [y,x,y,y,y]\\
  \hline
\end{array}
$$

Assume we are in Case 1, set $r=p^m$, where $m$ is the $p$-valuation of $\alpha^\gamma -1$. Imposing the defining
relations of $J$ we get in an essential way the following basic commutators. Since we work modulo the next term of the lower central series,
given $u,v\in\gamma_i$ we denote with $u\equiv v$ that $uv^{-1}\in\gamma_{i+1}$.
Let $a,b$ denote the generators of $J$ defined in Section \ref{s2}. Then $[b,a]=c^{-1}$ which we may substitute with $c$ as generators modulo $\gamma_3$ of $\gamma_2$.
Next $[b,a,a]=[c^{-1},a]\equiv [c,a]^{-1}=[a,c]=a^{-1}a^c=a^{r\ell}$, $p\nmid \ell$, which we may substitute $a^r$. Similarly from $[b,a,b]$ we obtain $b^r$.
In weight $4$, $[b,a,a,a]$ leads to $[a^r,b]\equiv c^r$, $[a,b,a,b]$ leads to $[a^r,b]\equiv [a,b]^r=c^r$ and $[b,a,b,b]\equiv 1$.
Finally, in weight~$5$, $[b,a,a,a,a]$ leads to $[c^r,a]\equiv [c,a]^r$ and to $a^{r^2}$, $[b,a,a,a,b]$ leads to $[c^r,b]\equiv [c,b]^r$ and to $b^{r^2}$,
$[b,a,a,b,b]$ leads to $[a^r,b,b]\equiv [c^r,b]$ and again $b^{r^2}$, $[b,a,b,b,a]$ and $[b,a,b,b,b]$ are both equivalent to $1$.
Since $a^{r^2}=b^{-r^2}$ we obtain the following table

$$
\begin{array}{|r|l|}
  \hline
  \text{Lower Term}&\text{Basic Generators}\\
  \hline
  \gamma_1 & a,b \\
  \gamma_2 & c \\
  \gamma_3 & a^r,b^r \\
  \gamma_4 & c^r \\
  \gamma_5 & a^{r^2}\\
  \hline
\end{array}
$$
and each section is isomorphic to $C_r\times C_r$, $C_r$, $C_r\times C_r$, $C_r$, $C_r$, respectively. Moreover since $\gamma_6=1$, $\gamma_3$ is abelian and equals 
$
\langle a^r\rangle\times \langle (ab)^r\rangle \times \langle c^r\rangle\simeq C_{r^2}\times C_r\times C_r.
$

\section{An altentative proof of the commutator formula}\label{andrea}

A new approach to the study of the Sylow $p$-subgroup $J$ of the Macdonald group $G(\alpha^\gamma,\alpha^\gamma)$ exploits 
a recent algorithm due to Eick and Cant (see \cite{CE}). Any finitely generated nilpotent torsion-free group $T$
admits a composition series 
$$
T=T_1\unrhd T_2\unrhd \ldots\unrhd T_n \unrhd T_{n+1}=1
$$
with $T_i/T_{i+1}$ infinite cyclic. The integer $n$ is uniquely determined and known under the name of \emph{Hirsch length}.
Given $a_i$'s such that $a_iT_{i+1}$ generates $T_i/T_{i+1}$, the structure of $T$ is uniquely determined by a cubic array 
$t\in \Z^{\binom{n}{3}}$ so that for $1\le i<j<k\le n$
$$
a_ja_i=a_ia_ja_{j+1}^{t_{ijj+1}}\cdots a_k^{t_{ijk}}\cdots a_n^{t_{ijn}}.
$$
One non-trivial task is to determine if a given array $t$ really defines a group of Hirsch length $n$.
Suppose for the remainder of this section that we are in Case 1. Then $J$ has nilpotency class $5$ and admits a series of subgroups
$$
J=J_1\unrhd J_2\ldots J_7\unrhd J_8=1
$$
with $J_i/J_{i+1}\simeq C_r$ cyclic of order $r=p^m$, where $v_p(\alpha^\gamma-1)=m$.
The hope is to be able to find a universal cover $T$ of $J$ which has the same commutator relations valid in $J$ independently from 
$\alpha,\gamma$ and $r$. We assume that $T$ is generated by two elements $a_1,a_2$ and set 
\begin{equation}\label{DefElts}
a_3=[a_2,a_1], a_4=[a_3,a_1], a_5=[a_3,a_2], a_6=[a_5,a_1], a_7=[a_6,a_2].
\end{equation}

From the previous paragraph we know that the free nilpotent group $F$ of rank $2$ and 
nilpotency class $5$ has lower central series sections of ranks $2,1,2,3$ and $6$. Hence its Hirsch length is $14$.
We need to force relations in $T$ emulating the fact that the lower central series sections of $J$ have ranks $2,1,2,1$ and $1$. So we need to kill
two basic commutators in weight $4$ and $5$ of them in weight $5$. In order to justify the previous definitions for the elements in $T$ we start by 
working out the elements corresponding to the $a_i$'s when we substitute $a_1,a_2$ with $a,b$, namely:

\begin{equation}\label{DefElts2}
a_3=[b,a], a_4=[b,a,a], a_5=[b,a,b], a_6=[b,a,b,a], a_7=[b,a,b,a,b].
\end{equation}

We recall that $\alpha^\gamma=1+r\ell$, where $r=p^m$ and $p\nmid \ell$. Let $v=-\ell+r\ell^2$ and $\mu=1+rv=1-r\ell+r^2\ell^2$.
Then $\alpha\mu\equiv 1\mod r^3$. For simplicity, we also set $\alpha_0=\alpha^\gamma$ and $\mu_0=\mu^\gamma$.

\begin{lemma}
  Given $i\in\Z$, $a^{b^i}$ equals $ab^{h_i}c^i$ for a suitable function $h$ exponential in $i$.
\end{lemma}

\begin{proof}\label{conj-a-by-b-powers}
  We first pin down $a^{b^i}$  for non-negative $i$. Clearly $a^{b^0}=a$ and $a^b=ac$. Assume that
 $$
 a^{b^i}=ab^{f_i}c^i.
 $$
 Then 
 $$
 a^{b^{i+1}}=a^bb^{f_i}(c^i)^b=acb^{f_i-1}c^{-1}c^{i+1}bc^{-(i+1)}c^{i+1}=ab^{\alpha_0(f_i-1) +\alpha_0^{i+1}}c^{i+1}.
 $$
 Therefore $f_{i+1}=\alpha_0(f_i-1) +\alpha_0^{i+1}$ is the solution of a first degree difference equation whose formal solution is
 $$
 f_i=\frac{(i-1)\alpha_0^{i+1}-i\alpha_0^i+\alpha_0}{\alpha_0-1}.
 $$
 We determine $f_i$ as a residue class mod $r^3$. Collecting $\alpha_0^i$, the numerator modulo $r^4$ becomes 
 $$
 (1+ir\ell+\binom{i}{2}r^2\ell^2+\binom{i}{3}r^3\ell^3)(-1+(i-1)r\ell)+1+r\ell\equiv
 \binom{i}{2}r^2\ell^2+\left(2\binom{i}{3}+\binom{i}{2}\right)r^3\ell^3\mod r^4.
 $$
Thus 
 $$
 f_i\equiv\binom{i}{2}r\ell+\left(2\binom{i}{3}+\binom{i}{2}\right)r^2\ell^2+\left(4 \binom{i}{4}+2 \binom{i}{3}\right)r^{3}\ell^{3}\mod r^3,
 $$
 where the last summand  may be discarded if $p>3$. 
 
 Assume that for non-negative $i$ 
 $$
 a^{b^{-i}}=ab^{g_i}c^{-i}.
 $$
 Conjugation by $b$ yields
 $$
 a^{b^{-i+1}}=acb^{g_i-1}c^{-1}c^{-(i-1)}bc^{i-1}c^{-i+1}=ab^{\alpha_0(g_i-1)+\mu_0^{i-1}}c^{-i+1}.
 $$
 Thus $g_{i-1}=\alpha_0(g_i-1)-\mu_0^{i-1}$ or, equivalently,
 $$
 g_i=\mu_0g_{i-1}-\mu_0^i+1,
 $$
 whose solution is 
 $$
 g_i=\frac{i\mu_0^{i+1}-(i+1)\mu_0^i+1}{1-\mu_0}.
 $$
 Again, as before, we get 
 \begin{equation}\label{E:gg function}
   g_i\equiv -\binom{i+1}{2}vr-2\binom{i+1}{3}v^2r^2-\left(4 \binom{i}{4}+3 \binom{i}{3}\right) r^{3} v^{3}\mod r^3,
 \end{equation}
 where the last summand  may be discarded if $p>3$. 
 \end{proof}
 
 Let $\delta_{a,b}$ denote the Kronecker delta, namely $\delta_{a,b}$ is $1$ if $a=b$, $0$ otherwise.
 We now proceed and determine explicitly the elements $a_i$ defined in (\ref{DefElts2}).

\begin{theorem}\label{Basic Commutators}
Set $a_1=a$ and $a_2=b$. Then the following identities hold in $J$: $a_3=[b,a]=c^{-1}$, $a_4=[b,a,a]=a^{-rv}$, $a_5=[b,a,b]=b^{-r\ell}$,
$a_6=[b,a,b,a]=c^{r\ell}b^{r^2(1-2\delta_{p,3}3^{m-1}\ell)v\ell /2}$
and $a_7=[b,a,b,a,b]=b^{-r^2\ell v}$.
\end{theorem}

\begin{proof}
By definition, $[b,a]=c^{-1}$, $[b,a,a]=ca^{-1}c^{-1}a=a^{-\mu_0+1}=a^{-rv}$, $[b,a,b]=cb^{-1}c^{-1}b=b^{-\alpha_0+1}=b^{-r\ell}$,
 $
 [b,a,b,a]=[b^{-r\ell},a]=(a^{b^{-r\ell}})^{-1}a=(ab^{g_{r\ell}}c^{-r\ell})^{-1}a=c^{r\ell}b^{-g_{r\ell}}.
 $
By (\ref{E:gg function}),
 $$
 -g_{r\ell}\equiv \binom{r\ell+1}{2}vr+2\binom{r\ell+1}{3}v^2r^2\equiv r^2v\ell/2-r^3v\ell^2/3\mod r^3.
 $$
 For $p>3$ the latter equals $r^2v\ell/2$. For $p=3$ we have an extra summand $-3^{m-1}r^2v\ell^2${\color{red}.}
 
Notice that since $r$ is odd, $v=-\ell+r\ell^2$ is always even, so the previous exponents are always integers. Therefore we deduce that $[b,a,b,a]$ equals
 $c^{r\ell}b^{r^2(1-2\delta_{p,3}3^{m-1}\ell)v\ell /2}$. Finally 
 $$[b,a,b,a,b]=[c^{r\ell}b^{r^2(1-2\delta_{p,3}3^{m-1}\ell)v\ell /2},b].$$
 Now $[c^{r\ell},b]=b^{1-\mu_0^{r\ell}}$. Since
 $$
 1-\mu_0^{r\ell}=1-(1+rv)^{r\ell}\equiv -r^2\ell v\mod r^3,
 $$
 the latter equals $b^{-r^2\ell v}$. Applying the well-known formula $[xy,z]=[x,z]^y[y,z]$ with $x=c^{r\ell}$, $y=b^{r^2\ell v/2}$ and $z=b$ the last claim is proved.
\end{proof}

We need to uncover commutator relations satisfied in $J$ and impose them 
on the free nilpotent group $F$ of rank $2$ and class $5$ in order to 
obtain a torsion-free nilpotent group $T$ of Hirsch length~$7$.
Following the recipe in \cite{CE} we need to determine a $7\times 7\times 7$ (cubic) array $t$ leading to such $T$.
We first set in GAP $t_{ijk}=0$ for all $1\le i,j,k\le 7$. These values have to be modified so that
$$
[a_j,a_i]=a_{j+1}^{t_{ijj+1}}\cdots a_k^{t_{ijk}}\cdots a_n^{t_{ijn}},
$$
for all $1\le i<j<k\le 7$ where the $a_i$'s are defined in (\ref{DefElts}).
The very definition forces
$$
t_{ijk}=1
$$ 
for $(i,j,k)=(1,2,3),(1,3,4),(2,3,5),(1,5,6)$, and $(2,6,7)$.

\begin{theorem}\label{T:tcoeffs}
  The other non-zero entries for $t$ are $t_{ijk}=1$ for $(i,j,k)=(1,6,7),(2,4,6),(3,4,7)$ and $t_{3,5,7}=-1$.
\end{theorem}

\begin{proof}
  When $i=1$ we only need to determine $[b,a,a,a]=[a^{-rv},a]=1$, so $t_{14k}=0$. For $i=2$, 
  $$
  [b,a,a,b]=[a^{-rv},b]=(b^{a^{-rv}})^{-1}b.
  $$
We apply the automorphism $\theta$ of $J$ sending $(a,b,c)$ to $(b,a,c^{-1})$ to the latter term and obtain
$$
(a^{b^{-rv}})^{-1}a=c^{rv}b^{-g_{rv}}.
$$
A calculation similar to the one applied to $g_{r\ell}$ furnishes $-g_{rv}\equiv r^2v^2/2\mod r^3$.
Thus the latter term equals $c^{rv}b^{r^2v^2/2}$ with $\theta$-preimage
$c^{-rv}a^{r^2v^2/2}=c^{r\ell}b^{r^2\ell v/2}$, since $v\equiv-\ell\mod r$ and $a^{r^2}=b^{-r^2}$.
Therefore $[b,a,a,b]=[b,a,b,a]$ and $t_{2,4,6}=1$.
Trivially $[b,a,b,b]=[b^{-r\ell},b]=1$ and $t_{2,5,k}=0$.
When $i=3$,
$$
[[b,a,a],[b,a]]=[a^{-rv},c^{-1}]=a^{rv(1-\mu_0)}=a^{-r^2v^2}=b^{-r^2v\ell}=[b,a,b,a,b]
$$
so $t_{3,4,7}=1$. Finally,
$$
[[b,a,b],[b,a]]=[b^{-r\ell},c^{-1}]=b^{r\ell(1-\alpha_0)}=b^{-r^2\ell^2}=[b,a,b,a,b]^{-1}
$$ so $t_{3,5,7}=-1$.
\end{proof}

We can now feed in this data into the GAP package HallPoly and obtain the Hall polynomials for $T$
namely integer-valued polynomials with rational coefficients $F_k$ such that given two $n$-tuples of integers
$x=(x_1,\ldots,x_n)$ and $y=(y_1,\ldots,y_n)$

\begin{equation}\label{E:HallPol}
  (a_1^{x_1}\ldots a_n^{x_n})(a_1^{y_1}\ldots a_n^{y_n})=a_1^{F_1(x,y)}\ldots a_n^{F_n(x,y)}.
\end{equation}

We then evaluate $F_k(x,y)$ for $x=(-i,-j,0,\ldots,0)$ and $y=(i,j,0,\ldots,0)$, $i,j\in\Z$.
We recall that rational coefficients integer-valued polynomials are integral combination of binomial polynomials, namely polynomials defined as
$$
\binom{z}{k}=\frac{z(z-1)\cdots(z-k+1)}{k!}
$$
for $0<k\in \N$, $\binom{z}{0}=1$. This result extends to more variables $z_1,\ldots,z_r$ considering integral combinations of multinomial polynomials 
$$
\binom{z_1}{k_1}\cdots\binom{z_r}{k_r}.
$$

The results read as follows:
\begin{itemize}
  \item $F_k=0$, for $k=1,2$;
  \item $F_3=-ij$;
  \item $F_4=-j\binom{i}{2}$;
  \item $F_5=-i\binom{j}{2}$;
  \item $F_6=-\binom{i}{2}\binom{j}{2}$.
\end{itemize}

The shape of $F_7(x,y)$ is more cumbersome and equals
$$
-ij^2/2 + i^3j^2/4 + i^2j^2/4 + ij/12 - i^2j^3/4 + i^3j/12 + ij^3/12.
$$
We need to convert it into an integral combination of multinomial polynomials for example introducing 
a graded lexicographic ordering $\succeq$. If we declare $i\succeq j$, the highest term in $F_7$
is $i^3j^2/4$ corresponding to the occurrence of $3\binom{i}{3}\binom{j}{2}$. We finally get
$$
F_7(x,y)=3\binom{i}{3}\binom{j}{2}-3\binom{i}{2}\binom{j}{3}+\binom{i}{2}\binom{j}{2}-i\binom{j}{3}+2j\binom{i}{3}+2j\binom{i}{2}-i\binom{j}{2}.
$$

We may finally state the following commutator formula for $J$.

\begin{theorem}
  Given $i,j\in\Z$, we have
  $$
  [a^i,b^j]=c^{ij}a^{-r\ell j\binom{i}{2}}b^{r\ell i\binom{j}{2}}c^{-r\ell\binom{i}{2}\binom{j}{2}}b^{-r^2v\ell\chi(i,j)/2},
  $$
  where $\chi(i,j)$ equals
  $$
  6\binom{i}{3}\binom{j}{2}-6\binom{i}{2}\binom{j}{3}+4\binom{i}{3}j+3\binom{i}{2}\binom{j}{2}-2i\binom{j}{3}+2\binom{i}{2}j-2i\binom{j}{2}+2\delta_{p,3}3^{m-1}\ell\binom{i}{2}\binom{j}{2}.                                                                                     
  $$
\end{theorem}

\begin{proof}
  By Theorem \ref{T:tcoeffs} and the determination of the Hall polynomials $F_k$ evaluated at $x=(-i,-j,0,\ldots,0)$ and $y=(i,j,0,\ldots,0)$, equation (\ref{E:HallPol}) becomes
  
  \begin{equation}
    [a_1^i,a_2^j]=a_3^{-ij}a_4^{-j\binom{i}{2}}a_5^{-i\binom{j}{2}}a_6^{-\binom{i}{2}\binom{j}{2}}a_7^{F_7}.\label{eq.comj.5}
  \end{equation}
  
Substituting the $a_i$ with the corresponding values in (\ref{DefElts2}) and exploiting Theorem \ref{Basic Commutators} we obtain 
  
  \begin{equation}
    [a^i,b^j]=c^{ij}a^{rv j\binom{i}{2}}b^{r\ell i\binom{j}{2}}(c^{r\ell}b^{r^2(1-2\delta_{p,3}3^{m-1}\ell)v\ell /2})^{-\binom{i}{2}\binom{j}{2}}b^{-r^2v\ell F_7}.\label{eq.comj.6}
  \end{equation}
We substitute $v$ with $-\ell+r\ell^2$ and move all powers of $b^{r^2}$ to the rightmost position and set $\chi(i,j)=-2j\binom{i}{2}+\binom{i}{2}\binom{j}{2}+2F_7$.
We keep $v$ in the exponent of the last factor because $v$ is always even.
\end{proof}

Since $r=p^m$, $a^{c^{ij}}=a^{\alpha_o^{ij}}$, $b^{c^{ij}}=b^{\mu_o^{ij}}$ and $a^{r^2}=b^{-r^2}\in Z(J)$,
one can easily deduce that (\ref{eq.comj.4}) is equivalent to (\ref{eq.comj.5}).

\section{Case 2 when $\alpha=\beta$}\label{s5}

We keep the notation and hypotheses from Sections \ref{s2} and \ref{s3}, and assume further 
that $p=2$. 

In relation to (\ref{defrelgp}), \cite[Theorem 7.1 and Proposition 7.2]{MS} ensure that $o(a)=o(b)=2^{3m-1}$,
$o(c)=2^{2m}$, and $|J|=2^{7m-3}$. 
Thus $J$ is generated by  $a$, $b$, and $c$ subject to the defining relations
\begin{equation}
\label{relgpdos}
c=[a,b], a^{c}=a^{\alpha^\gamma},\, {}^{c} b=b^{\alpha^\gamma}, a^{2^{3m-1}}=1, b^{2^{3m-1}}=1.
\end{equation}
Notice the existence of the automorphism $a\leftrightarrow b$, $c\leftrightarrow c^{-1}$ of $J$.

\begin{theorem}\label{fern}
    For all $i,j\in\Z$ the following commutator formulas hold in $J$:
    \begin{align}
        &[c^i,a^j]
        = a^{-2^m\ell ij - 2^{2m}\ell^2\phi(i)j},\label{eq.fer}\\
        &[c^i,b^j]
        = b^{2^m\ell ij - 2^{2m}\ell^2\phi(i+1)j},\label{eq.fer2}\\
        &[a^i,b^j]
        =a^{-2^m\ell\phi(i)j} b^{2^m\ell i\phi(j)} c^{ij - 2^m\ell\phi(i)\phi(j)} a^{\xi(i,j)},\label{eq.fer3}
				\end{align}
    where 
    \[
    \xi(i,j)= 2^{2m-1}\ell^2\{2\varphi(i+1)j + (2i-7)\phi(i)\phi(j) - 2i\phi(j) - (3i+1)i\varphi(j)\}.
    \]
		\end{theorem}

This is stated in \cite[Theorem 3.1]{MS2} and proven in  \cite[Theorem 8.1]{MS2}.

Recall that $\alpha^\gamma=1+2^m\ell$, where $m\geq 1$ and $2\nmid \ell$;
$\gamma=2^n k$, where $n\geq 0$, $2\nmid k$; $\alpha=1+2^h t$, where $h\geq 0$, and $2\nmid t$. Recall also that
$N_p$ is the normal closure of the $V_i$, $i\geq 1$, defined by (\ref{defvi}) and that by Proposition \ref{adentro}, we have
$N_2\subseteq Z_3(J)$.



Due to \cite[Proposition 6.4]{MS} the following fundamental relations hold in $J$:
\begin{equation}
\label{funrel2}
a^{2^{2m-1}}b^{2^{2m-1}}=1,\; a^{3^{3m-2}}=c^{2m-1}=b^{2^{3m-2}}.
\end{equation}
Moreover, \cite[Theorem 8.1 and Proposition 8.2]{MS} give
\begin{equation}
\label{zjdos}
Z(J)=\langle a^{2^{2m-1}}\rangle=\langle b^{2^{2m-1}}\rangle, 
Z_2(J)=\langle a^{2^{2m-1}}, c^{2^{m-1}}\rangle,\text{ and if }m>1,
Z_3(J)=\langle a^{2^{m}}, b^{2^{m}}, c^{2^{m-1}}\rangle.
\end{equation}
Here $Z_3(J)$ is abelian if $m>1$ by \cite[Proposition 9.1]{MS}. Furthermore, by \cite[Proposition 7.2]{MS},
\begin{equation}
\label{inter12}
\langle a,c\rangle\cap \langle b\rangle=Z(J)=\langle b,c\rangle\cap \langle a\rangle,
\end{equation}
\begin{equation}
\label{inter22}
\langle a\rangle\cap \langle c\rangle=\langle c^{2^{2m-1}}\rangle=\langle a^{2^{3m-2}}\rangle=\langle b\rangle\cap \langle c\rangle.
\end{equation}

\begin{prop}\label{muevoandrea} Suppose $m>1$. Then
$$
Z_3(J)=\langle a^{2^{m}}, a^{2^{2m-2}}c^{2^{m-1}}, a^{2^{m}}b^{2^{m}}\rangle=
\langle b^{2^{m}}, b^{2^{2m-2}}c^{2^{m-1}}, a^{2^{m}}b^{2^{m}}\rangle
\cong C_{2^{2m-1}}\times  C_{2^{m}}\times C_{2^{m-1}},
$$
where $Z_3(J)$ is the internal direct product of $\langle a^{2^{m}}\rangle$ , $\langle a^{2^{2m-2}}c^{2^{m-1}}\rangle$, 
and $\langle a^{2^{m}}b^{2^{m}}\rangle$.
\end{prop}

\begin{proof} Use that (\ref{funrel2})-(\ref{inter22}), that $Z_3(J)$ is abelian, and the stated orders of $a$, $b$, and $c$.
\end{proof}

As in Case 1, we see that $\alpha^i\equiv 1\equiv \mu^i\mod 2^h$, $i\geq 1$, and  
\begin{equation}
\label{evvv}
v_2(\alpha-1)=h=v_2(\mu-1).
\end{equation}

Since $\alpha^i\mu^i\equiv 1\mod 2^{3m-1}$, we find that (\ref{defvi}), (\ref{zjdos}), and Theorem \ref{fern} give
\begin{equation}
\label{vidos}
V_i=a^{-2^{m-1}(\alpha^i-1)\ell} b^{2^{m-1}(\mu^i-1)\ell} c^{-2^{m-2}(\alpha^i-1)(\mu^i-1)\ell} z_i,\quad i\geq 1,
\end{equation}
where $z_i\in Z(J)$.

\medskip

\noindent{\bf The case $h\geq 2$.} We assume throughout this subsection that $h\geq 2$. Then $m=h+n$.
Applying Theorem \ref{fern}, we see that $z_i^{2^n}=1$ in (\ref{vidos}). Also,
as $Z_3(J)$ is abelian, Proposition \ref{adentro} together with 
$\alpha^i\equiv 1\equiv\mu^i\mod 2^h$, $c^{2^{2m}}=1$, $m=h+n$, $z_i^{2^n}=1$,
(\ref{zjdos}), and (\ref{vidos}) yield
\begin{equation}
\label{vipdos}
V_i^{2^n}=a^{-2^{m+n-1}(\alpha^i-1)\ell} b^{2^{m+n-1}(\mu^i-1)\ell}\in Z(J).
\end{equation}

Exactly as in Case 1, we see that $N_2$
is contained in the normal subgroup $M_{n}\cap L_{2n}$ of $J$, where $V\in M_n$ (resp. $V\in L_{2n}$)
if $V\in Z_3(J)$ and $V^{2^{n}}\in Z(J)$ (resp. $V^{2^{2n}}=1$), and that

\begin{prop}\label{v2ndos} If $2\nmid i$, then $V_i$ has order $2^{2n}$.
\end{prop}

\begin{prop}\label{muevoandrea2} We have
$$
M_n=\langle a^{2^{m+h-1}}, a^{2^{2m+h-2}}c^{2^{m+h-1}}, a^{2^{m+h-1}}b^{2^{m+h-1}}\rangle=
\langle b^{2^{m+h-1}}, b^{2^{2m+h-2}}c^{2^{m+h-1}}, a^{2^{m+h-1}}b^{2^{m+h-1}}\rangle.
$$
\end{prop}

\begin{proof} Immediate consequence of Proposition \ref{muevoandrea} and the definition of $M_n$.
\end{proof}

\begin{prop}\label{muevoandrea3} We have
$$
L_{2n}=\begin{cases} \langle a^{2^{m+2h-1}}, a^{2^{2m-2}}c^{2^{m-1}},a^{2^{m}}b^{2^{m}}\rangle=
\langle b^{2^{m+2h-1}}, b^{2^{2m-2}}c^{2^{m-1}},a^{2^{m}}b^{2^{m}}
\rangle\text{ if }n\geq h,\\
\langle a^{2^{m+2h-1}}, a^{2^{m+2h-2}}c^{2^{2h-1}},a^{2^{2h-1}}b^{2^{2h-1}}\rangle=\langle 
b^{2^{m+2h-1}}, b^{2^{m+2h-2}}c^{2^{2h-1}},a^{2^{2h-1}}b^{2^{2h-1}}
\rangle\text{ if }n<h.
\end{cases}
$$
\end{prop}

\begin{proof} Immediate consequence of Proposition \ref{muevoandrea} and the definition of $L_{2n}$.
\end{proof}

\begin{prop}\label{mlndos} We have
$$
\begin{aligned}
M_n\cap L_{2n} &=\langle a^{2^{m+2h-1}}, a^{2^{2m+h-2}}c^{2^{m+h-1}},a^{2^{m+h-1}}b^{2^{m+h-1}}\rangle\\
&=\langle b^{2^{m+2h-1}}, b^{2^{2m+h-2}}c^{2^{m+h-1}},a^{2^{m+h-1}}b^{2^{m+h-1}}\rangle.
\end{aligned}
$$
Thus, $M_n\cap L_{2n}$ is the internal direct product of 
$\langle a^{2^{m+2h-1}}\rangle$, $\langle a^{2^{2m+h-2}}c^{2^{m+h-1}}\rangle$,$\langle a^{2^{m+h-1}}b^{2^{m+h-1}}\rangle$,
as well as the internal direct product of  
$\langle b^{2^{m+2h-1}}\rangle$, $\langle b^{2^{2m+h-2}}c^{2^{m+h-1}}\rangle$,$\langle a^{2^{m+h-1}}b^{2^{m+h-1}}\rangle$,
and it is isomorphic to $C_{2^{2n}}\times  C_{2^{n}}\times C_{2^{n}}$.
\end{prop}

\begin{proof} Immediate consequence of Propositions \ref{muevoandrea2} and \ref{muevoandrea3}.
\end{proof}

\begin{prop}\label{ordnzdos} We have $\Omega_n(Z(J))=\langle a^{2^{2m+h-1}}\rangle=\langle  b^{2^{2m+h-1}}\rangle\subseteq N_2$.
In particular, all $z_i$ appearing in (\ref{vidos})  are in $N_2$.
\end{prop}

\begin{proof} We know from \eqref{vipdos} that
$V_1^{2^n}\in Z(J)$. By Proposition \ref{v2ndos}, the order of $V_1^{2^n}$ is~$2^n$.
Thus $\langle V_1^{2^n}\rangle$ is a subgroup of $Z(J)$ of order $2^n$.
Being cyclic, $Z(J)$ has only one subgroup of order $2^n$, namely $\langle a^{2^{2m+h-1}}\rangle=\langle  b^{2^{2m+h-1}}\rangle$.
Thus $N_p\supseteq \langle V_1^{2^n}\rangle=\langle a^{2^{2m+h-1}}\rangle=\Omega_n(Z(J))$.
\end{proof}

\begin{prop}
\label{nor12} We have $c^{2^{m+h-1}}a^{2^{2m+h-2}}\in N_2$ and $c^{2^{m+h-1}}b^{2^{2m+h-2}}\in N_2$,
where these elements have order $2^n$.
\end{prop}

\begin{proof} As $V_1\in Z_3(J)$, we have $V_1^a\equiv V_1\mod Z_2(J)$. Carefully using Theorem \ref{fern} we find that
$$
V_1^a=V_1 c^{2^{m+h-1}x}b^{2^{2m+h-2}y}u,
$$
where $x,y\in\Z$ are odd and $u\in\langle a^{2m+h-1}\rangle$. It follows from Proposition \ref{ordnzdos} that
$$
c^{2^{m+h-1}}b^{2^{2m+h-2}e}\in N_2,
$$
for some odd integer $e$. But $a^{2m+h-1}\in N_2$ and $2m+h-2\geq 2m$, so dividing $e$ by 2 and using (\ref{funrel2}) yields
$$
c^{2^{m+h-1}}a^{2^{2m+h-2}}\in N_2, c^{2^{m+h-1}}b^{2^{2m+h-2}}\in N_2.
$$ 
It now follows from (\ref{funrel2}) and (\ref{inter22}) that the stated orders are correct.
\end{proof}

\begin{cor}
\label{potc} We have $c^{2^{m+h}}\in N_2$. In particular, all $c^{-2^{m-2}(\alpha^i-1)(\mu^i-1)\ell}$ 
appearing in (\ref{vidos})  are in $\langle c^{2^{m+h}}\rangle$
and hence in $N_2$.
\end{cor}

\begin{proof} By Proposition \ref{nor12}, $c^{2^{m+h-1}}a^{2^{2m+h-2}}\in N_2$, whence its square 
$c^{2^{m+h}}a^{2^{2m+h-1}}$ is also in $N_2$. It now follows from Proposition \ref{ordnzdos} that
$c^{2^{m+h}}\in N_2$.  

Since $\alpha^i\equiv 1\mod 2^h$, $\mu^i\equiv 1\mod 2^h$, we have $2^{m+2h-2}\mid 2^{m-2}(\alpha^i-1)(\mu^i-1)$. As $h\geq 2$,
we deduce $m+2h-2\geq m+h$. Therefore
$c^{-2^{m-2}(\alpha^i-1)(\mu^i-1)\ell}$ is in $\langle c^{2^{m+h}}\rangle$.
\end{proof}


\begin{prop}
\label{mazda2} $N_2$ is the normal closure of $\{F_i\,|\, i\geq 1\}\cup\{a^{2^{2m+h-1}}, c^{2^{m+h}}\}$, 
with $F_i$ as in~(\ref{deffi}).
\end{prop}

\begin{proof} This follows from (\ref{vidos}) and Propositions \ref{ordnzdos} and \ref{potc}.
 \end{proof}

We have $\alpha=1+2^h t$, so that $\mu=1-2^h t+2^{2h} f$, where $f$ is odd. Thus $\mu=1+2^h t_0$,
where $t_0=-t+2^h f$ is odd, and  $\alpha=1-2^h t_0+2^{2h}f_0$, where $f_0$ is odd. 

Using (\ref{evvv}) and $v_2(\mu-\alpha)=h+1$,
we deduce from (\ref{vui}) with $i=2$ that
\begin{equation}
\label{mi1}
a^{2^{m+2h}}\in N_2,\; b^{2^{m+2h}}\in N_2,
\end{equation}
and from (\ref{vui2}) that
\begin{equation}
\label{mi2}
F_i\in \langle F_1, a^{2^{m+2h}}\rangle,\; F_i\in \langle F_1, b^{2^{m+2h}}\rangle,\quad i\geq 1.
\end{equation}
On the other hand, taking $i=1$ in (\ref{deffi}) yields
$$
F_1=a^{-2^{m+h-1}t} b^{-2^{m+h-1}t}b^{2^{m+2h-1}f}=a^{-2^{m+h-1}t} b^{-2^{m+h-1}t}b^{-2^{m+2h-1}t}b^{2^{m+2h-1}(f+t)}.
$$
As $t$ is odd, $f+t$ is even, and $b^{2^{m+2h}}\in N_2$, we deduce
\begin{equation}
\label{mi3}
a^{2^{m+h-1}} b^{2^{m+h-1}(1+2^h)}\in N_2,\; F_1\in \langle a^{2^{m+h-1}} b^{2^{m+h-1}(1+2^h)}, b^{2^{m+2h}} \rangle.
\end{equation}
Likewise, from
$$
F_1=a^{2^{m+h-1}t_0} b^{2^{m+h-1}t_0}a^{-2^{m+2h-1}f_0}=a^{2^{m+h-1}t_0} b^{2^{m+h-1}t_0}a^{2^{m+2h-1}t_0}a^{-2^{m+2h-1}(f_0+t_0)},
$$
we infer 
\begin{equation}
\label{mi4}
a^{2^{m+h-1}(1+2^h)} b^{2^{m+h-1}}\in N_2,\; F_1\in \langle a^{2^{m+h-1}(1+2^h)} b^{2^{m+h-1}}, a^{2^{m+2h}} \rangle.
\end{equation}

If $n=0$ then $h>0$ and Theorem \ref{tres} imply $W_2\cong G_2$, with $G_2$ described in detail in \cite{MS}.
Thus, we assume for the remainder of this subsection that $n>0$.

Consider the subgroup $U$ of $N_2$, defined by
$$
\begin{aligned}
U &=\langle a^{2^{m+h-1}(1+2^h)} b^{2^{m+h-1}}, a^{2^{m+2h}}, a^{2^{2m+h-2}}c^{2^{m+h-1}}\rangle\\
&=\langle a^{2^{m+h-1}} b^{2^{m+h-1}(1+2^h)}, b^{2^{m+2h}}, b^{2^{2m+h-2}}c^{2^{m+h-1}}\rangle.
\end{aligned}
$$
It follows easily from Proposition \ref{mlndos} that 
$$
U\cong C_{2^{2n-1}}\times C_{2^n}\times C_{2^n},\; [M_n\cap L_{2n}:U]=2.
$$
Moreover, a careful application of Theorem \ref{fern} shows that $U$ is normal in $J$. We deduce from
Propositions  \ref{nor12} and \ref{mazda2}, as well as (\ref{mi1}), (\ref{mi2}),  (\ref{mi3}), and (\ref{mi4}),
that $N_2=U$. 

\begin{theorem}\label{main2} Suppose $n>0$. Then
$$
\begin{aligned}
N_2 &=\langle a^{2^{m+h-1}(1+2^h)} b^{2^{m+h-1}}, a^{2^{m+2h}}, a^{2^{2m+h-2}}c^{2^{m+h-1}}\rangle\\
&=\langle a^{2^{m+h-1}} b^{2^{m+h-1}(1+2^h)}, b^{2^{m+2h}}, b^{2^{2m+h-2}}c^{2^{m+h-1}}\rangle
\cong C_{2^{2n-1}}\times C_{2^n}\times C_{2^n};
\end{aligned}
$$
$$
|N_2|=2^{4n-1},\; |G_2/N_2|=2^{7h+3n-2},\; |W_2|=2^{7h+4n-2}=2^{7m-3n-2};
$$
$G_2/N_2$ is generated by elements $a_0,b_0,c_0$ subject to defining relations (\ref{rel0}),
\begin{equation}
\label{pel1}
a_0^{2^{m+2h}}=1=b_0^{2^{m+2h}}, a_0^{2^{2m+h-2}}c_0^{2^{m+h-1}}=1=b_0^{2^{2m+h-2}}c_0^{2^{m+h-1}},
\end{equation}
\begin{equation}
\label{pel2}
a_0^{2^{m+h-1}(1+2^h)}b_0^{2^{m+h-1}}=1=a_0^{2^{m+h-1}}b_0^{2^{m+h-1}(1+2^h)},
\end{equation}
and every element of $G_2/N_2$ can be written in one and only way in the form
\begin{equation}
\label{write2}
a_0^{e_1} b_0^{e_2}c_0^{e_3},\quad 0\leq e_1<2^{m+2h}, 0\leq e_2<2^{m+h-1}, 0\leq e_3<2^{m+h-1};
\end{equation}
$W_2$ is generated by elements $a_0,b_0,c_0,d_0$ subject to the defining relations (\ref{rel0}), (\ref{pel1}), (\ref{pel2}),
and
\begin{equation}
\label{pel3}
a_0^{d_0}=a_0^{\alpha^{qk}},\; {}^{d_0} b_0=b_0^{\alpha^{qk}},\; d_0^{2^n}=c_0^q,
\end{equation}
and every element of $W_2$ can be written in one and only way in the form
$$
a_0^{e_1} b_0^{e_2}c_0^{e_3}d_0^{e_4},\quad 0\leq e_1<2^{m+2h}, 0\leq e_2<2^{m+h-1}, 0\leq e_3<2^{m+h-1}, 0\leq e_4<2^{n}.
$$
Moreover, the automorphism $a\leftrightarrow b$, $c\leftrightarrow c^{-1}$ of $G_2$ induces an automorphism $a_0\leftrightarrow b_0$, 
$c_0\leftrightarrow c_0^{-1}$ on $G_2/N_2$, which extends to an automorphism of $W_2$ via $d_0\leftrightarrow d_0^{-1}$.
\end{theorem}

\begin{proof} All statements up to and including $2^{4n-1}$ have already been proven. It now follows from \cite[Theorem 7.1]{MS} that
$
|G_2/N_2|=2^{7h+7n-3}/2^{4n-1}= 2^{7h+3n-2},
$
whence by Theorem \ref{dos}
$
|W_2|=2^{7h+4n-2}.
$
We see from $N_2=\langle a^{2^{m+h-1}(1+2^h)} b^{2^{m+h-1}}, a^{2^{m+2h}}, a^{2^{2m+h-2}}c^{2^{m+h-1}}\rangle$ and (\ref{relgpdos}) that
$G_2/N_2$ has the stated presentation. It follows from \cite[Lemma 6.1 or Theorem 7.1]{MS} that every element
of $G_2/N_2$ can the written as stated in (\ref{write2}). Uniqueness follows from the fact that 
$|G_2/N_2|=2^{7h+3n-2}$. The remaining statements about $W_2$ now follow from Theorem \ref{dos} and $|W_2|=2^{7h+4n-2}$.
\end{proof}

\medskip

\noindent{\bf The case $h=1$ and $t\neq -1$.} We assume throughout this subsection that $h=1$,  $t\neq -1$,
and $n>0$ (if $n=0$ then $h>0$ and Theorem \ref{tres} imply $W_2\cong G_2$, with $G_2$ described in \cite{MS}).
Thus $\alpha=1+2t$, where $2\nmid t$. Note that $m=g+n+1$, where $g=v_2(t+1)$.
We have $\mu=1-2 t+ 4 f$, with $f$ odd. Hence $\mu=1+2 t_0$,
where $t_0=-t+2 f$ is odd, and  $\alpha=1-2t_0+4f_0$, where $f_0$ is odd.

According to (\ref{defui}) and (\ref{vidos}), we have
\begin{equation}
\label{cot}
V_i=a^{-2^{m-1}(\alpha-1)u_i\ell}b^{2^{m-1}(\mu-1)v_i\ell}c^{-2^{m-2}(\alpha-1)(\mu-1)u_iv_i\ell}a^{2^{2m-1}(u_iv_i+2k_i)},
\end{equation}
for some integer $k_i$. Making use (\ref{funrel2}), we see  that
\begin{equation}
\label{cop}
V_i^{2^{m-1}}=a^{-2^{2m-2}\ell[(\alpha-1)u_i+(\mu-1)v_i]}c^{-2^{2m-3}(\alpha-1)(\mu-1)u_iv_i\ell}a^{2^{3m-2}(u_iv_i+2k_i)}.
\end{equation}

Taking $i=1$ in (\ref{cop}) and appealing to (\ref{funrel2}), we find that
$$
V_1^{2^{m-1}}=a^{-2^{2m-2}\ell[(\alpha-1)+(\mu-1)]}.
$$
Here $(\alpha-1)+(\mu-1)=4f$. As $f$ is odd, we see that $V_1^{2^{m-1}}=a^{2^{2m}x}$, where $x$ is odd.
Then
\begin{equation}
\label{was}
a^{2^{2m}}\in N_2
\end{equation}
and $V_1$ has order $2^{2m-2}$. Using Theorem \ref{fern} together with (\ref{funrel2}) and  (\ref{cop}) shows that
$$
V_1^a=V_1 c^{2^m e} b^{2^{2m-1} d},
$$
where $e$ and $d$ are odd integers. Thus $c^{2^m} b^{2^{2m-1} g}\in N_2$ for some odd integer $g$.
Here $b^{2^{2m}}\in N_2$ by (\ref{funrel2}) and (\ref{was}), which yields that 
\begin{equation}
\label{was2}
c^{2^m} b^{2^{2m-1}},\; c^{2^m} a^{2^{2m-1}},\; c^{2^{m+1}},\;a^{2^{2m}},\; b^{2^{2m}}\in N_2.
\end{equation}
Going back to (\ref{cot}) and noting that $u_iv_i\ell$ and
$u_iv_i+2k_i$ are both even or both odd, we deduce that 
\begin{equation}
\label{was3}
c^{-2^{m-2}(\alpha-1)(\mu-1)u_iv_i\ell}a^{2^{2m-1}(u_iv_i+2k_i)}\in N_2\cap \langle a^{2^{2m-1}}c^{2^m}, c^{2^{m+1}}, a^{2^{2m}}\rangle
\cap \langle b^{2^{2m-1}}c^{2^m}, c^{2^{m+1}}, b^{2^{2m}}\rangle,
\end{equation}
and therefore
$$
F_i=a^{-2^{m-1}(\alpha-1)u_i}b^{2^{m-1}(\mu-1)v_i}\in N_2,\quad i\geq 1.
$$

Recall that
$
v_2(\alpha-1)=1=v_2(\mu-1)
$
and
$$
v_2(\alpha-\mu)=v_2(\alpha^2-1)=v_2((\alpha-1)(\alpha+1))=v_2((\alpha-1)2(1+t))=2+g.
$$
Thus, in the notation of (\ref{defui}), we have
$$
v_2(2^{m-1}(\alpha-1)w_i)=m+2+g=v_2(2^{m-1}(\mu-1)w_i).
$$
Taking $i=2$ in (\ref{vui}), we deduce 
\begin{equation}
\label{was4}
a^{2^{2m+1-n}}\in N_2,\; b^{2^{2m+1-n}}\in N_2,
\end{equation}
while (\ref{vui2}) yields
\begin{equation}
\label{was5}
F_i\in \langle F_1, a^{2^{2m+1-n}}\rangle,\; F_i\in \langle F_1, b^{2^{2m+1-n}}\rangle,\quad i\geq 1.
\end{equation}
It now follows from (\ref{was}), (\ref{was2}), (\ref{was3}), (\ref{was4}), and (\ref{was5}) that
$N_2$ is the normal closure in $J$ of each of the following subsets:
$
\{F_1, a^{2m+1-n}, a^{2^{2m-1}}c^{2^m}\},\;\{F_1, b^{2m+1-n}, b^{2^{2m-1}}c^{2^m}\}.
$
But
$$
F_1=a^{-2^{m-1}(\alpha-1)} b^{2^{m-1}(\mu-1)}=a^{-2^{m-1}(\alpha-1)} b^{2^{m-1}(\alpha-1)} b^{2^{m-1}[(\mu-1)-(\alpha-1)]},
$$
where $v_2(\alpha-\mu)=2+g$. Since $b^{2^{2m+1-n}}\in N_2$, we deduce that
$$
a^{-2^m}b^{2^m}b^{2^{2m-n}}\in N_2,\; F_1\in \langle a^{-2^m}b^{2^m}b^{2^{2m-n}}, b^{2^{2m+1-n}}\rangle.
$$
It follows that $N_2$ is the  normal closure in $J$ of $\{a^{-2^m}b^{2^{m}(1+2^{m-n})}, a^{2m+1-n}, a^{2^{2m-1}}c^{2^m}\}$ as well as of 
$\{a^{2^{m}(1+2^{m-n})}b^{-2^m}, b^{2m+1-n}, b^{2^{2m-1}}c^{2^m}\}$. The subgroups generated by these subsets
are easily seen to be normal by means of Theorem \ref{fern}. Thus
$$
N_2=\langle a^{-2^m}b^{2^{m}(1+2^{m-n})}, a^{2m+1-n}, a^{2^{2m-1}}c^{2^m}\rangle=
\langle a^{2^{m}(1+2^{m-n})} b^{-2^m}, b^{2m+1-n}, b^{2^{2m-1}}c^{2^m}\rangle,
$$
This description of $N_2$ and Theorem \ref{muevoandrea} yield 
$N_2\cong C_{2^{m+n-2}}\times C_{2^{m-1}}\times C_{2^{m-1}}$. We have proven the following result.

\begin{theorem}\label{main3} The group
$$
N_2=\langle a^{-2^m}b^{2^{m}(1+2^{m-n})}, a^{2m+1-n}, a^{2^{2m-1}}c^{2^m}\rangle=
\langle a^{2^{m}(1+2^{m-n})}b^{-2^m}, b^{2m+1-n}, b^{2^{2m-1}}c^{2^m}\rangle
$$
is isomorphic to  $C_{2^{m+n-2}}\times C_{2^{m-1}}\times C_{2^{m-1}}$;
$$
|N_2|=2^{3m+n-4},\; |G_2/N_2|=2^{4m-n+1},\; |W_2|=2^{4m+1};
$$
$G_2/N_2$ is generated by elements $a_0,b_0,c_0$ subject to defining relations (\ref{rel0}),
\begin{equation}
\label{tel1}
a_0^{2^{2m+1-n}}=1=b_0^{2^{2m+1-n}}, a_0^{2^{2m-1}}c_0^{2^{m}}=1=b_0^{2^{2m-1}}c_0^{2^{m}},
\end{equation}
\begin{equation}
\label{tel2}
a_0^{2^{m}(1+2^{m-n})}b_0^{2^{m}}=1=b_0^{2^{m}(1+2^{m-n})}a_0^{2^{m}},
\end{equation}
and every element of $G_2/N_2$ can be written in one and only way in the form
\begin{equation}
\label{write3}
a_0^{e_1} b_0^{e_2}c_0^{e_3},\quad 0\leq e_1<2^{2m+1-n}, 0\leq e_2<2^{m}, 0\leq e_3<2^{m};
\end{equation}
$W_2$ is generated by elements $a_0,b_0,c_0,d_0$ subject to the defining relations (\ref{rel0}), (\ref{tel1}), (\ref{tel2}),
and
\begin{equation}
\label{tel3}
a_0^{d_0}=a_0^{\alpha^{qk}},\; {}^{d_0} b_0=b_0^{\alpha^{qk}},\; d_0^{2^n}=c_0^q,
\end{equation}
and every element of $W_2$ can be written in one and only way in the form
$$
a_0^{e_1} b_0^{e_2}c_0^{e_3}d_0^{e_4},\quad 0\leq e_1<2^{2m+1-n}, 0\leq e_2<2^{m}, 0\leq e_3<2^{m}, 0\leq e_4<2^{n}.
$$
Moreover, the automorphism $a\leftrightarrow b$, $c\leftrightarrow c^{-1}$ of $G_2$ induces an automorphism $a_0\leftrightarrow b_0$, 
$c_0\leftrightarrow c_0^{-1}$ on $G_2/N_2$, which extends to an automorphism of $W_2$ via $d_0\leftrightarrow d_0^{-1}$.
\end{theorem}

In particular, if $W=W_1(3,3,2)=W_2$ we have $n=1$, and $m=3$, and $W$ has order $2^{13}$, in agreement with the order
given in \cite{HNO}.

\medskip

\noindent{\bf The case $\alpha=-1$, $n=0$.} Assume that $\alpha=-1=\alpha^\gamma$.
Then $G_2\cong Q_{16}$, the quaternion group of order 16, by \cite[Proposition 8.2]{MS}. 
Moreover, $G=G_2$ by \cite[Section 4]{M}, as the only prime factor
of $\alpha^\gamma-1$ is 2. As $\gamma$ is odd, Theorem \ref{tres} gives $W_2\cong G\cong Q_{16}$.

\medskip

\noindent{\bf The case $h=0$.} This impossible, since $\alpha^\gamma\equiv 1\mod 2^m$,
with $m\geq 1$, forces $\alpha$ to be odd.

\section{Commutator calculations in Case 3}\label{s6}

We keep the notation and hypotheses from Sections \ref{s2} and \ref{s3}, and assume further that $p=3$
and $\alpha^\gamma\equiv 7\mod 9$. Thus $\alpha^\gamma=1+3\ell$, where 
$\ell=-1+3g$, $g\in\Z$, and we set $\alpha_0=\alpha^\gamma$.

By \cite[Theorem 7.1 and Proposition 7.2]{MS}, we know that $a$ and $b$ have order 81, $c$ has order 27,
and $J$ has order $3^{10}$. By 
\cite[Theorem 8.1]{MS}, we have
$$Z(J)=\langle a^{27}\rangle, Z_2(J)=\langle a^{27},c^9\rangle, Z_3(J)=\langle a^9,b^9,c^9\rangle=J^9,
Z_4(J)=\langle a^9,b^9,c^3\rangle, 
$$
$$
Z_5(J)=\langle a^3,b^3,c^3\rangle=J^3,
Z_6(J)=\langle a^3,b^3,c\rangle, Z_7(J)=J.
$$
By Proposition \ref{adentro}, we know that $N_3\subseteq Z_5(J)$,
which has order $3^7$ by \cite[Proposition 9.3]{MS}, which also ensures that 
$Z_3(J)$ is abelian of order~$3^4$. Since $\alpha_0^3\equiv 1\mod 9$,
it follows that $Z_4(J)$ is also abelian, necessarily of
order $3^5$, as $|Z_4(J)/Z_3(J)|=3$ by the above uniqueness of expression. 
Furthermore, by \cite[Proposition 6.4]{MS}, we have the fundamental relation:
\begin{equation}
\label{futres}
a^{27}b^{27}=1.
\end{equation}
By \cite[Theorem 5.3]{MS}, $J$ has presentation
$$
\langle a,b\,|\, a^{[a,b]}=a^\alpha,\, b^{[b,a]}=b^\alpha, a^{81}=1, b^{81}=1\rangle,
$$
which yields the automorphism $a\leftrightarrow  b$, $c\leftrightarrow  c^{-1}$, of $J$.  Since
$$
(1+3\ell)^i\equiv 1+3i\ell+9\phi(i)\ell^2-27\varphi(i)\mod 81,\quad i\in\N,
$$
we easily obtain the following result.

\begin{prop}\label{prop.commutators3}
    For all $i,j\in\Z$ the following commutator formulas hold in $J$:
    \begin{align}
        &[c^i,a^j]
        = a^{-3\ell ij - 9\ell^2\phi(i)j+27\varphi(i)j},\label{eq.c31}\\
        &[c^i,b^j]
        = b^{3\ell ij - 9\ell^2\phi(i+1)j-27\varphi(i+2)j}.\label{eq.c32}
    \end{align}
    \end{prop}
In particular,
\begin{equation}
\label{cab}
[c^3,a^9]=1=[c^3,b^9].
\end{equation}

An analogue of (\ref{eq.comj.4}), namely the calculation of $[a^r,b^s]$ for arbitrary integers $r,s$, is needed to determine $N_3$
and requires substantial work. Dividing $r$ and $s$ by 3, the computation of $[a^r,b^s]$
translates into the calculation of all commutators of the form
\begin{equation}
\label{goal}
[a^{e+3i},b^{f+3j}],
\end{equation}
where $e,f\in\{-1,0,1\}$ and $i,j\in\Z$. To deal with (\ref{goal}), 
recall that for an arbitrary group $T$, and $x,y,z\in T$, we have
\begin{equation}
\label{commu}
[x,yz]=[x,z][x,y]^z,\; [yz,x]=[y,x]^z\; [z,x].
\end{equation}
Thus by (\ref{commu}), 
\begin{equation}
\label{exp1}
[a^{e+3i},b^{f+3j}]=[a^{e},b^{f+3j}]^{a^{3i}}[a^{3i},b^{f+3j}],
\end{equation}
\begin{equation}
\label{exp2}
[a^{e},b^{f+3j}]^{a^{3i}}=[a^{e},b^{3j}]^{a^{3i}}[a^{e},b^{f}]^{b^{3j}a^{3i}},
\end{equation}
\begin{equation}
\label{exp3}
[a^{3i},b^{f+3j}]=[a^{3i},b^{3j}][a^{3i},b^{f}]^{b^{3j}},
\end{equation}
and therefore
\begin{equation}
\label{delta5}
[a^{e+3i},b^{f+3j}]=[a^{e},b^{3j}]^{a^{3i}}[a^{e},b^{f}]^{b^{3j}a^{3i}}[a^{3i},b^{3j}][a^{3i},b^{f}]^{b^{3j}}.
\end{equation}
We proceed to compute all commutators
\begin{equation}
\label{exp4}
[a^{3i},b^{3j}],
\end{equation}
\begin{equation}
\label{exp5}
[a^{e},b^{3j}]^{a^{3i}},
\end{equation}
\begin{equation}
\label{exp6}
[a^{3i},b^{f}]^{b^{3j}},
\end{equation}
\begin{equation}
\label{exp7}
[a^{e},b^{f}]^{b^{3j}a^{3i}},
\end{equation}
as well as the right hand of (\ref{delta5}), which yields (\ref{goal}).

We begin by finding an explicit formula for $[a,b^i]$ for an arbitrary integer $i$.
From $a^c=a^{\alpha_0}$ and ${}^c b=b^{\alpha_0}$ we easily obtain by induction
$$
a^{b^i}=a b^{(\alpha_0-1)(\alpha_0+2\alpha_0^2+\cdots+(i-1)\alpha_0^{i-1})} c^i,\quad i\geq 2,
$$
whence
\begin{equation}
\label{con3}
[a,b^i]=b^{(\alpha_0-1)(\alpha_0+2\alpha_0^2+\cdots+(i-1)\alpha_0^{i-1})} c^i,\quad i\geq 2.
\end{equation}
Recalling that $b^{81}=1$, $\alpha_0\equiv 1\mod 3$, and $\phi(j)=j(j-1)/2$, we note that
\begin{equation}
\label{upa}
\alpha_0^j\equiv (1+(\alpha_0-1)^j)\equiv 1+(\alpha_0-1)j+(\alpha_0-1)^2 \phi(j)\mod 27,\quad j\geq 0.
\end{equation}
Set 
$$
f(i)=1^2+\cdots+(i-1)^2\text{ for }i\geq 2,\text{ and }g(i)=2\phi(2)+\cdots+(i-1)\phi(i-1)\text{ for }i\geq 3,
$$ 
and further define $g(2)=0$. Then by (\ref{upa}),
\begin{equation}
\label{upa2}
\alpha_0+2\alpha_0^2+\cdots+(i-1)\alpha_0^{i-1}\equiv \phi(i)+ (\alpha_0-1)f(i)+
(\alpha_0-1)^2g(i)\mod 27,\quad i\geq 2.
\end{equation}
Since $\ell^3\equiv -1\mod 3$, we deduce from (\ref{con3}) and (\ref{upa2}) that
\begin{equation}
\label{con4}
[a,b^i]=b^{3\ell\phi(i)+ 9\ell^2 f(i)-27g(i)} c^i,\quad i\geq 2.
\end{equation}
Further setting $f(0)=0=g(0)$ and $f(1)=0=g(1)$ we readily verify that
(\ref{con4}) holds for all $i\geq 0$, where
$\phi(i)=\binom{i}{2}$, $f(i)=2\binom{i}{3}+\binom{i}{2}$ and $g(i)=3\binom{i}{4}+2\binom{i}{2}$ for all $i\geq 0$.
Let $s$ be any integer. Then
$$
\phi(i+54s)\equiv \phi(i)\mod 27,\quad i\in\Z.
$$
Set
$$
F(i)=(i-1)i(2(i-1)+1)/6=(i-1)i(2i-1)/6,\quad i\in\Z,
$$
so that $F(i)=f(i)$ for $i\geq 0$ and $F(i)\equiv F(i+54s)\mod 9$ for $i\in\Z$.

For any integer $i$, we have $\phi(i)\equiv 0\mod 3$ if $i\equiv 0\mod 3$ or $i\equiv 1\mod 3$, and $\phi(i)\equiv 1\mod 3$
if $i\equiv 2\mod 3$. Thus, for $i\geq 2$, $g(i)$ is congruent modulo 3 to the product of 2 multiplied by the amount of
integers between 2 and $i-1$, including the endpoints, that are congruent to 2 modulo~3. This amount is one more than 
the integral part of $(i-3)/3$.
Thus, for $i\geq 2$, $g(i)$ is congruent to $-\{1+[(i-3)/3]\}=-[i/3]$ modulo 3. Set
$$
G(i)=[i/3],\quad i\in\Z,
$$
so  that $G(i)\equiv -g(i)\mod 3$ for $i\geq 0$ and $G(i)\equiv G(i+54s)\mod 3$ for $i\in\Z$. Thus,
$$
[a,b^i]=b^{3\ell\phi(i)+ 9\ell^2 F(i)+27G(i)} c^i,\quad i\geq 0.
$$
If $i<0$, select $s\in\Z$ so that $i+162s\geq 0$. Then
$$
[a,b^i]=[a,b^{i+162s}]=b^{3\ell\phi(i+162s)+ 9\ell^2 F(i+162s)+27G(i+162s)} c^{i+162 s}=
b^{3\ell\phi(i)+ 9\ell^2 F(i)+27G(i)} c^i.
$$
Therefore, the following commutator formula is valid for all integers $i$:
\begin{equation}
\label{con44}
[a,b^i]=b^{3\ell\phi(i)+ 9\ell^2 F(i)+27G(i)} c^i.
\end{equation}
Applying the automorphism $a\leftrightarrow b$ of $J$, we deduce 
\begin{equation}
\label{con55}
[b,a^i]=a^{3\ell\phi(i)+ 9\ell^2 F(i)+27G(i)} c^{-i},\quad i\in\Z.
\end{equation}
We next compute (\ref{exp4}). From (\ref{con44}), we derive
$[a,b^3]=b^{9\ell+45\ell^2+27}c^{3}$.
Using $\ell=-1+3g$, we obtain
\begin{equation}
\label{conq}
[a,b^3]=b^{63}c^{3}.
\end{equation}
Hence by (\ref{commu}),
$
[a^2,b^3]=[a,b^3]^a[a,b^3]=(b^{36}c^{3}b^{27})^a b^{63}c^{3},
$
where $b^{27}\in Z(J)$. Thus,
$$
[a^2,b^3]=(b^{36})^a (c^3)^a b^9 c^3.
$$
Here by (\ref{eq.c31}) and (\ref{con44}),
$$
(b^{9})^a=b^{9}[b^{9},a]=b^{9}[a,b^{9}]^{-1}=b^{-18} c^{-9},\; (b^{36})^a=b^{9} c^{-9},\;
(c^3)^a=c^{3}[c^3,a]=c^3 a^{-9\ell}.
$$
Therefore
$$
[b^2,a^3]=b^{9} c^{-9} c^3 a^{-9\ell} b^{9}c^{3}=a^{-9\ell}b^{18}c^{-3},
$$
as all factors belong to the abelian group $Z_4(J)$.  Thus,
$$
[a^3,b^3]=[a^2,b^3]^a[a,b^3]=(a^{-9\ell}b^{18} c^{-3})^a b^{63}c^{3}=b^{27}c^9.
$$
Hence
\begin{equation}
\label{333}
[a^3,b^3]=b^{27}c^9=a^{-27}c^9.
\end{equation}
Since $[a^3,b^3]$ commutes with both $a^3$ and $b^3$, it follows from (\ref{commu}) and (\ref{333}) that
\begin{equation}
\label{666}
[a^{3i},b^{3j}]=a^{-27ij}c^{9ij}=b^{27ij}c^{9ij},\quad i,j\in\Z.
\end{equation}
This takes care of (\ref{exp4}). Note, in particular, that
\begin{equation}
\label{999}
[a^9,b^3]=1=[a^3,b^9].
\end{equation}

We next compute  (\ref{exp7}) when $e=1=f$, that is $[a,b]^{b^{3j}a^{3i}}=c^{{b^{3j}a^{3i}}}$.
The relation ${}^c b=b^{\alpha_0}$ gives
$$
[a,b]^{b^{3j}}=b^{-3j}c b^{3j} c^{-1} c=b^{-3j} b^{3j(1+3\ell)} c=b^{9j\ell}c.
$$
By (\ref{999}), $b^9$ commutes with $a^3$, so
$
[a,b]^{b^{3j}a^{3i}}=b^{9j\ell }c^{a^{3i}}.
$
Here the relation $a^c=a^{\alpha_0}$ yields
$$
c^{a^{3i}}=cc^{-1} a^{-3i}c a^{3i}=c a^{-3i(1+3\ell)} a^{3i}=c a^{-9i\ell}.
$$
The inverse of $1+3\ell$ modulo 9 is $1-3\ell$ and $\ell^2\equiv 1\mod 3$, so
$$
ca^{-9i\ell}=c a^{-9i\ell}c^{-1} c=a^{-9i\ell(1-3\ell)}c=a^{-9i\ell-27i}c.
$$
As $b^9$ commutes with $a^9$, we infer
\begin{equation}
\label{prie}
[a,b]^{b^{3j}a^{3i}}=a^{-9i\ell-27i}b^{9j\ell}c.
\end{equation}

We next compute (\ref{exp5}) when $e=1$, that is, $[a,b^{3i}]^{a^{3j}}$. From (\ref{con44}), we deduce
$$
[a,b^{3i}]=b^{3\ell\phi(3i)+ 9\ell^2 F(3i)+27G(3i)} c^{3i}=b^{9i\ell[1+ \ell(6i-1)](3i-1)/2+27i} c^{3i}.
$$
Using $\ell=-1+3g$, we see that
$
i\ell[1+ \ell(6i-1)](3i-1)/2\equiv i(3i+1)\mod 9,
$
so
$$
[a,b^{3i}]=b^{9i+27i(i+1)} c^{3i},\quad i\in\Z.
$$
Since $a^{3}$ commutes with $b^9$, 
$
[a,b^{3i}]^{a^{3j}}=b^{9i+27i(i+1)} (c^{3i})^{a^{3j}}.
$
Here by (\ref{eq.c31}),
$$
(c^{3i})^{a^{3j}}=c^{3i}[c^{3i},a^{3j}]=c^{3i}a^{27ij}.
$$
We thus obtain
\begin{equation}
\label{tante1}
[a,b^{3i}]^{a^{3j}}=b^{9i+27i(i-j+1)} c^{3i},\quad i,j\in\Z.
\end{equation}
This takes care of (\ref{exp5}) when $e=1$. We proceed to compute (\ref{exp5}) when $e=-1$, that is, $[a^{-1}, b^{3i}]^{a^{3j}}$. 
By (\ref{tante1}),
$
a^{b^{3i}}=a b^{9i+27i(i+1)} c^{3i}.
$
Since $b^9$ and $c^3$ commute, we infer
$$
(a^{-1})^{b^{3i}}=b^{-9i-27i(i+1)} c^{-3i} a^{-1}.
$$
Here by (\ref{eq.c31}),
$
c^{-3i} a^{-1}=a^{-1}c^{-3i}[c^{-3i}, a^{-1}]=a^{-1}c^{-3i}a^{-9\ell i},
$
so
$
(a^{-1})^{b^{3i}}=b^{-9i-27i(i+1)} a^{-1} c^{-3i} a^{-9\ell i},
$
that is
$
(a^{-1})^{b^{3i}}=a^{-9\ell i+27i(i+1)}b^{-9i}a^{-1} c^{-3i}.
$
We next determine $[a^{-1},b^{9}]$. From (\ref{con44}), we obtain
$
[a,b^9]=c^9,
$
so (\ref{eq.c31}) gives
$
[a,b^9]^{a^{-1}}=(c^9)^{a^{-1}}=c^9[c^9,{a^{-1}}]=c^9a^{-27}=a^{-27}c^9.
$
Thus by (\ref{commu}),
$
1=[aa^{-1},b^9]=[a,b^9]^{a^{-1}}[a^{-1},b^9],
$
which yields
$
[a^{-1},b^9]=([a,b^9]^{a^{-1}})^{-1}=a^{27}c^{-9}.
$
Since  $a^{27}c^{-9}$ commutes with $b^9$, we deduce
$
[a^{-1},b^{9i}]=a^{27i}c^{-9i},\; [b^{9i},a^{-1}]=a^{-27i}c^{9i},\; [b^{-9i},a^{-1}]=a^{27i}c^{-9i}.
$
Therefore
$
(a^{-1})^{b^{3i}}=a^{-9\ell i+27i(i+1)}a^{-1}b^{-9i}a^{27i}c^{-9i} c^{-3i}=
a^{-1-9\ell i+27i(i+2)}b^{-9i}c^{-12i},
$
which implies
\begin{equation}
\label{duro}
[a^{-1}, b^{3i}]=a^{-9\ell i+27i(i+2)}b^{-9i}c^{-12i},\quad i\in\Z.
\end{equation}
Since $a^3$ commutes with $b^9$, we infer
$$
[a^{-1}, b^{3i}]^{a^{3j}}=
a^{-9\ell i+27i(i+2)}b^{-9i}(c^{-12i})^{a^{3j}}=a^{-9\ell i+27i(i+2)}b^{-9i}c^{-12i}[c^{-12i},a^{3j}].
$$
Here (\ref{eq.c31}) gives
$
[c^{-12i},a^{3j}]=a^{-27ij}.
$
We thus obtain
\begin{equation}
\label{duro2}
[a^{-1}, b^{3i}]^{a^{3j}}=a^{-9\ell i+27i(i-j+2)}b^{-9i}c^{-12i}.
\end{equation}
This yields (\ref{exp5}) when $e=-1$ and completes the determination of (\ref{exp5}).
Now  (\ref{exp6}) follows immediately by applying the automorphism $a\leftrightarrow b$ of $J$
to (\ref{tante1}) and (\ref{duro2}), and then inverting, which presents no difficulty as all factors commute.
We have
\begin{equation}
\label{tante2}
[a^{3i},b]^{b^{3j}}=a^{-9i-27i(i-j+1)} c^{3i},\quad i,j\in\Z,
\end{equation}
\begin{equation}
\label{duro3}
[a^{3i},b^{-1}]^{b^{3j}}=a^{9i}b^{9\ell i-27i(i-j+2)}c^{-12i}.
\end{equation}


We next calculate (\ref{exp7}) when $e=-1=-f$. We first compute $[a^{-1},b^{-1}]$. From (\ref{con44}), we have
$$
a^{b^i}=a b^{3\ell\phi(i)+ 9\ell^2 F(i)+27G(i)} c^i,\quad i\in \Z,
$$
whence
$$
(a^{-1})^{b^i}=c^{-i} b^{-3\ell\phi(i)- 9\ell^2 F(i)-27G(i)} a^{-1},\quad i\in \Z.
$$
Taking $i=-1$ yields
$$
(a^{-1})^{b^{-1}}=c b^{-3\ell+9\ell^2+27} a^{-1},
$$
$$
[a^{-1},b^{-1}]=a(a^{-1})^{b^{-1}}=ac b^{-3\ell+9\ell^2+27} a^{-1}.
$$
As ${}^c b=b^{1+3\ell}$ and $\ell^3\equiv -1\mod 3$, we infer
$$
[a^{-1},b^{-1}]=ab^{(-3\ell+9\ell^2+27)(1+3\ell)} c a^{-1}=ab^{-3\ell} c a^{-1}.
$$
Now $a^c=a^{1+3\ell}$ and $\ell\equiv -1\mod 3$, so the inverse of $1+3\ell$ modulo 81 is $1-3\ell+9\ell^2+27$, whence 
${}^c (a^{-1})=a^{-1+3\ell-9\ell^2-27}$. Thus
by (\ref{999}),
$$
[a^{-1},b^{-1}]=ab^{-3\ell} a^{-1+3\ell-9\ell^2-27} c=a^{1-9\ell^2-27} b^{-3\ell} a^{3\ell}a^{-1} c.
$$
Here by (\ref{333}),
$$
[b^{-3\ell}, a^{3\ell}]=[a^{3\ell},b^{-3\ell}]^{-1}=[a^{3\ell},b^{3\ell}]=a^{-27\ell^2}c^{9\ell^2}=
a^{-27}c^{9},
$$
so
$$
[a^{-1},b^{-1}]=a^{1-9\ell^2-27} a^{3\ell}b^{-3\ell}[b^{-3\ell}, a^{3\ell}] a^{-1} c=
a^{1+3\ell-9\ell^2-54}b^{-3\ell}c^9 a^{-1} c.
$$
As $c^9=c^{-18}$ and $\ell\equiv -1\mod 3$, we have $(1+3\ell)^{18}\equiv 1+27\mod 81$, so
$$
[a^{-1},b^{-1}]=
a^{1+3\ell-9\ell^2-54}b^{-3\ell}a^{-27} a^{-1} c^{10}=a^{1+3\ell-9\ell^2}b^{-3\ell}a^{-1} c^{10}.
$$
Here by (\ref{duro}), $\ell\equiv -1\mod 3$, and the fact that $Z_4(J)$ is abelian, we deduce
$$
[b^{-3\ell}, a^{-1}]=[a^{-1},b^{-3\ell}]^{-1}=[a^{-1},b^{3\ell}]=
a^{-9\ell^2+27\ell(\ell+1)}b^{-9\ell}c^{-12\ell}=a^{-9\ell^2}b^{-9\ell}c^{-12\ell},
$$
so by (\ref{999}) we finally obtain
\begin{equation}
\label{duro4}
[a^{-1},b^{-1}]=a^{3\ell-18\ell^2}b^{-12\ell}c^{10-12\ell}.
\end{equation}

We can now calculate (\ref{exp7}) when $e=-1=-f$. As $[a^9,c^9]=[b^3,c^9]=[a^9,b^3]=1$,
(\ref{duro4}) yields
$$
[a^{-1},b^{-1}]^{b^{3j}}=a^{3\ell-18\ell^2}[a^{3\ell},b^{3j}]b^{-12\ell}c^{10-12\ell}[c^{1-3\ell}, b^{3j}].
$$
Using (\ref{eq.c32}), (\ref{666}), $\ell\equiv -1\mod 3$, and ${}^c b=b^{1+3\ell}$, we obtain
$$
[a^{-1},b^{-1}]^{b^{3j}}=a^{3\ell-18\ell^2}b^{-3\ell-54j}
c^{1-3\ell +9(1+j-\ell)}.
$$
Thus 
$$
[a^{-1},b^{-1}]^{b^{3j}a^{3i}}=a^{3\ell-18\ell^2}b^{-3\ell-54j}[b^{-3\ell-54j},a^{3\ell-18\ell^2}] 
c^{1-3\ell +9(1-j-\ell)}[c^{1-3\ell},a^{3i}].
$$
Using (\ref{futres}), (\ref{eq.c31}), (\ref{666}), $\ell\equiv -1\mod 3$, and ${}^c a=a^{1-3\ell+9\ell^2+27}$, we obtain
\begin{equation}
\label{duro5}
[a^{-1},b^{-1}]^{b^{3j}a^{3i}}=a^{3\ell-9\ell i-18\ell^2-27j}b^{-3\ell}
c^{1-3\ell +9(1+i+j-\ell)}.
\end{equation}

We next compute  (\ref{exp7}) when $(e,f)=(1,-1)$ and $(e,f)=(-1,1)$, that is,
$[a,b^{-1}]^{b^{3j}a^{3i}}$ and $[a^{-1},b]^{b^{3j}a^{3i}}$. This is not needed to study $N_3$, it just completes (\ref{goal}),
so we merely state the results:
$$
[a,b^{-1}]^{b^{3j}a^{3i}}=a^{9\ell i+27(-i-j+1)}b^{3\ell-9\ell(\ell+j)}c^{-1+9i},\; 
[a^{-1},b]^{b^{3j}a^{3i}}=a^{-3\ell+9i\ell} b^{-9 j\ell} c^{-1+9j}.
$$

To study $N_3$ we require explicit formulas for (\ref{goal}) when $e=1=f$ and $e=-1=f$. More precisely, we need
$$
[a^{1+3i},b^{1+3j}][a,b]^{-1},\;  [a^{-1+3i},b^{-1+3j}][a,b]^{-1},
$$
where $i,j\in\Z$ are subject to certain conditions. We will compute both cases in general, beginning with 
$[a^{1+3i},b^{1+3j}][a,b]^{-1}$. 

We apply (\ref{delta5}), (\ref{666}), (\ref{prie}), (\ref{tante1}) subject to $i\leftrightarrow j$, 
and (\ref{tante2}). All resulting factors commute, except that we need to use
$$
c a^{-9i}=c a^{-9i} c^{-1}c=a^{-9i(1-3\ell)} c=a^{-9i-27i}c.
$$
Gathering like terms and making use of $\ell=-1+3g$ and (\ref{futres}), we obtain
\begin{equation}
\label{n3uno}
[a^{1+3i},b^{1+3j}][a,b]^{-1}=a^{-27(j-ij+i^2+j^2+(i+j)g)}c^{3(i+j)+9ij}.
\end{equation}

We next compute $[a^{-1+3i},b^{-1+3j}][a,b]^{-1}$. We apply (\ref{delta5}), (\ref{666}), (\ref{duro2})
subject to $i\leftrightarrow j$, (\ref{duro3}), 
and (\ref{duro5}). All  resulting factors commute, except that we need to use
$$
c^{-3j}a^{3\ell}=a^{3\ell}c^{-3j}[c^{-3j},a^{3\ell}]=a^{3\ell}c^{-3j}a^{27j},
$$
$$
c^{-3j}b^{-3\ell}=b^{-3\ell}c^{-3j}[c^{-3j},b^{-3\ell}]=b^{-3\ell}c^{-3j}b^{27j},
$$
$$
{}^c (b^{9\ell i})=b^{9\ell i(1+3\ell)}=b^{9\ell i+27i},\; {}^c (a^{9i})=a^{9i(1-3\ell)}a^{9 i+27i},
$$
which produces the central term $a^{27j}b^{27j}b^{27i}a^{27i}=1$ when writing (\ref{delta5}) in normal form. We obtain
\begin{equation}
\label{n3dos}
[a^{-1+3i},b^{-1+3j}][a,b]^{-1}=a^{3\ell-9[\ell(i+j)-i+2\ell^2]+27(i^2+j^2+2i+j)}
b^{-3\ell+9(\ell i-j)}c^{-3(i+j+\ell)+9(1-ij-\ell)}.
\end{equation}

\section{Case 3 when $\alpha=\beta$}\label{s7}

We keep the notation and hypotheses from Sections \ref{s2}, \ref{s3}, and \ref{s6}. 
Recall that $\alpha_0=\alpha^\gamma=1+3\ell$, which implies $\alpha\equiv 1\mod 3$ or $\alpha\equiv -1\mod 3$.
Here $\ell=-1+3g$, with $g\in\Z$, so $\alpha_0\equiv 7\mod 9$.

If $\alpha\equiv 1\mod 3$ then $\alpha=1+3u$, in which case $\alpha^\gamma\equiv 7\mod 9$ means $u\gamma\equiv 2\mod 3$.
In particular, $3\nmid\gamma$. But $3|(\alpha-1)$. Therefore, by Theorem \ref{tres}, $W_3\cong G_3$,  with $G_3$
described in detail in \cite{MS}. We will ignore this case in what follows and assume below that $\alpha=-1+3u$, in which
case $\alpha^\gamma\equiv 7\mod 9$ means $\gamma\equiv 0\mod 2$ and $u\gamma\equiv 1\mod 3$.
In particular, $3\nmid\gamma$ and  $3\nmid u$.


Note that as $\alpha=-1+3u$ and $\alpha\mu\equiv 1\mod 81$, we have $$\mu\equiv -1-3u-9u^2-27u^3\mod 81,$$ and we
set $v=-(u+3u^2+9u^3)$, so that $\mu\equiv -1+3v\mod 81$. Thus
$$
\alpha^e\equiv (-1)^e+(-1)^{e-1} 3 eu+(-1)^{e-2} 9\phi(e) u^2\mod 27,\quad e\geq 2
$$
$$
\mu^e\equiv (-1)^e+(-1)^{e-1} 3 ev+(-1)^{e-2} 9 \phi(e) v^2\mod 27, \quad e\geq 2.
$$
Recalling that $a^{27},b^{27}\in Z(J)$, it follows that for all $e\geq 2$, we have
\begin{equation}
\label{formulavi}
V_e=[a^{\alpha^e},b^{\mu^e}][a,b]^{-1}=[a^{(-1)^e+(-1)^{e-1} 3 eu+(-1)^{e-2}9\phi (e) u^2},
b^{(-1)^e+(-1)^{e-1} 3 ev+(-1)^{e-2}9\phi(e) v^2}][a,b]^{-1}.
\end{equation}
In particular,
$
V_1=[a^{-1+3u},
b^{-1+3v}][a,b]^{-1}
$ 
and
$
V_2=[a^{1-6u+9u^2},b^{1-6v+9v^2}][a,b]^{-1}.
$
First take $e$ even in (\ref{formulavi}) and apply (\ref{n3uno})  with $i=-eu+3\phi(e)u^2$
and $j=-ev+3\phi(e)v^2$,  and $v=-(u+3u^2+9u^3)$. It follows from (\ref{n3uno})
that 
$
V_e=a^{-27eu}.
$
Thus all $V_e$, $e$ even, are central. Since $3\nmid u$, taking $e=2$ we get $V_2=a^{27u}$,
so $Z(J)\subseteq N_3$. 

Next take $e$ odd in (\ref{formulavi}) and apply (\ref{n3dos})  with $i=-eu+3\phi(e)u^2$,
$j=-ev+3\phi(e)v^2$,  and $v=-(u+3u^2+9u^3)$. It follows from (\ref{n3dos})
that 
$
V_e=a^{3\ell} b^{-3\ell}c^{-3\ell}a^{9f_1}b^{9f_2}c^{9f_3},
$
for some integers $f_1,f_2,f_3$. Thus
$
V_e=a^{3x} b^{3y}c^{3z},
$
for some integers $x,y,z$ none of which is a multiple of 3 and such that $x+y\equiv 0\mod 3$.

\begin{prop}\label{poop} Let $e\in\Z$ be odd. Then the normal closure of $V_e$ in $J$ is the abelian
subgroup
$$
T=\langle a^3b^{-3}, a^9, b^9,c^3\rangle=\langle a^3b^{-3}, c^3, a^9b^9\rangle=
\langle a^3b^{-3}\rangle\times \langle c^3\rangle\times \langle a^9b^9\rangle
\cong C_{27}\times C_9\times C_3. 
$$
Moreover, $T=N_3$ contains $Z_4(J)$ and is contained in $Z_5(J)$.  
\end{prop}

\begin{proof} It is clear that $Z_4(J)\subseteq T\subseteq Z_5(J)$. As $Z_5(J)/Z_4(J)$ is a free module of rank
2 over $\Z/\Z_3$, it follows from (\ref{futres}) that $|T|=3^6$. In any group $R$, a group lying between consecutive terms
of the upper central series of $R$ is automatically normal in $R$. Thus $T$ is normal in $J$. 
We already know that $a^9$ and $b^9$ commute with each other and
as well as with $c^3$, $a^3$, and $b^3$, and hence with $a^3b^{-3}$. As $(1+3\ell)^3\equiv 1+9\ell\mod 27$,
with inverse $1-9\ell$ modulo 27, we see from (\ref{futres}) that
$$
(a^3b^{-3})^{c^3}=a^{3(1+9\ell)}b^{-3(1-9\ell)}=a^{3-27}b^{-3-27}=a^3b^{-3}.
$$
This shows that $T$ is abelian. Our prior analysis ensures that all $V_f$ are in $T$. As $T$
is a normal subgroup of $J$, we infer $N_3\subseteq T$. We proceed to show that
the normal closure of $V_e$, say $Q$, is $T$. As $V_e\in N_3$, it will then follow that $T\subseteq N_3$,
whence $T=N_3$. As indicated above, we have
$$
V_e=a^{3x} b^{3y}c^{3z},
$$
for some integers $x,y,z$ none of which is a multiple of 3 and such that $x+y\equiv 0\mod 3$. Here $V_e\in Q$.
As $T$ is normal in $J$, it follows that $Q\subseteq T$. Let us prove the reverse inclusion. We have
$$
V_e^a=a^{3x} b^{3y} [b^{3y},a] c^{3z} [c^{3z},a],
$$
where $[b^{3y},a]=[a,b^{3y}]^{-1}=b^{-9\ell y(1-\ell)}c^{-3y}$ by (\ref{con44}) and $[c^{3z},a]=a^{-9\ell z}d$
for some $d\in Z(J)$ by~(\ref{eq.c31}). Thus 
$$[V_e,a]=V_e^{-1} V_e^a=a^{-9\ell z}b^{-9\ell y(1-\ell)}c^{-3y}d=a^{9f_1}b^{9f_2}c^{3f_3}d\in Q,
$$
where none of $f_1,f_2,f_3$ are multiples of 3. Likewise, (\ref{eq.c31}), (\ref{eq.c32}), (\ref{con44}), and (\ref{con55}) imply that
$$
[V_e,a,a]=a^{9f_4}c^{9f_5}d_2\in Q, [V_e,a,b]=b^{9f_5}c^{9f_6}d_3\in Q,
$$
where $d_2,d_3\in Z(J)$ and none of $f_4,f_5,f_6,f_7$ are multiples of 3. Now
(\ref{eq.c32}) and (\ref{con55}) yield that
$$
[V_e,a,a,b]=c^{9f_8}d_3\in Q,
$$
where $d_4\in Z(J)$ and $3\nmid f_8$. We finally deduce from (\ref{eq.c31}) that
$$
[V_e,a,a,b,a]=a^{27 f_9}\in Q,
$$
where $3\nmid f_9$. We successively deduce that the following elements belong to $Q$:
$$
a^{27}, c^{9}, a^{9}, b^9, c^3, a^3b^{-3}.
$$
This proves that $T\subseteq Q$. From (\ref{333}) we get $(a^3 b^{-3})^3=a^9 b^{-9} a^{-54} c^{18}$.
Thus $o(a^3 b^{-3})=27$, $(a^3 b^{-3})^3 a^9 b^9 (c^3)^{-6}=a^{-36}$, and
$T=\langle a^3b^{-3}, c^3, a^9b^9\rangle$. It follows from \cite[Theorem 7.1]{MS} that 
$T=\langle a^3b^{-3}\rangle\times \langle c^3\rangle\times \langle a^9b^9\rangle\cong C_{27}\times C_9\times C_3$.
\end{proof}

\begin{theorem}\label{poop2} The group
$
N_3=\langle a^3b^{-3}, a^9, b^9,c^3\rangle,
$
is group of order $3^6$. Moreover, $W_3\cong G_3/N_3$ is the sole Sylow 3-subgroup of $W$, and is
the group of order 81 generated by elements $a_0,b_0,c_0$ subject to the defining
relations:
$$
a_0^{c_0}=a_0^{-2}, b_0^{c_0}=b_0^{4}, a_0^9=b_0^9=c_0^3=1, a_0^3=b_0^3, [a_0,b_0]=c_0.
$$
\end{theorem}

\begin{proof} This is an immediate consequence of Proposition \ref{poop} and our prior work on $G_3$.
\end{proof}

\section{The order of $W$}

Recall that $\alpha$ and $\gamma$ are integers such that $\gamma>0$ and $\alpha^\gamma\neq 1$.
In particular, if $\alpha=-1$, then $2\nmid\gamma$. Given a prime number $p$,
we set $m=v_p(\alpha^\gamma-1)$, $h=v_p(\alpha-1)$, $n=v_p(\gamma)$. Note if $p=2$, $m>0\Leftrightarrow h>0$.
When $mn>0$, we define $v(p,\alpha,\gamma)$
to be: $n$ if $n>0$ and $m=0$; $7m-3n$ if $\alpha\equiv 1\mod p$, $p$ is odd, and if $p=3$, then $\alpha^\gamma\not\equiv 7\mod 9$; 
$4m$ if $\alpha\equiv -1\mod p$, $p$ is odd, and if $p=3$, then $\alpha^\gamma\not\equiv 7\mod 9$;
$3m+n$ if $\alpha\not\equiv \pm 1\mod p$, $p$ is odd, and if $p=3$, then $\alpha^\gamma\not\equiv 7\mod 9$;
$7m-3$ if $m>0$, $n=0$, and $p=2$; $7m-3n-2$ if $h>1$, $n>0$, and $p=2$; $4m+1$ if $h=1$, $n>0$, $\alpha\neq -1$, and $p=2$; 
10 if $p=3$, $\alpha^\gamma\equiv 7\mod 9$, and $\alpha\equiv 1\mod 3$; 4 if $p=3$, $\alpha^\gamma\equiv 7\mod 9$, and $\alpha\equiv -1\mod 3$.

As a consequence of Theorems \ref{one}, \ref{tres}, \ref{main1}, \ref{hzero}, \ref{hzero2}, \ref{main2}, \ref{main3}, \ref{poop2}, and 
\cite[Theorem 7.1]{MS}, we have the following result.

\begin{theorem}\label{orderW} The order of $W$ is
$|W|=\underset{p\mid (\alpha^\gamma-1)\gamma}\Pi p^{v(p,\alpha,\gamma)}$.
\end{theorem}

\section{The nilpotency class of $W_p$}

Applying Theorems \ref{tres}, \ref{main1}, \ref{hzero}, \ref{hzero2}, \ref{main2}, \ref{main3}, and \ref{poop2},
as well as the commutator formulas from Theorems \ref{thm.commutators.J} and \ref{fern}, we readily
obtain the following result.

\begin{theorem}\label{nilp} The nilpotency class $c$ of $W_p$ is as follows:

$\bullet$ Case 1, $h>n\geq 0$. Then $c=5$, with
$$\gamma_{2}=\langle a_0^{p^{h}}, b_0^{p^{h}}, c_0\rangle,\; 
\gamma_{3}=\langle a_0^{p^{m}}, b_0^{p^{m}}, c_0^{p^{h}}\rangle,\;
\gamma_{4}=\langle a_0^{p^{m+h}},c_0^{p^{m}}\rangle,\;
\gamma_{5}=\langle a_0^{p^{2m}}\rangle, \gamma_{6}=1.
$$

$\bullet$ Case 1, $0<h\leq n$. Then $\gamma_{i+1}=\langle a_0^{p^{i h}}, b_0^{p^{i h}}, c_0^{p^{(i-1)h}}\rangle$, $i\geq 1$,
so $c=\lceil m/h\rceil+2$.

$\bullet$ Case 1, $\alpha\not\equiv \pm 1\mod p$. For $s=v_p(\alpha^k-1)$, we have
$
1\equiv \alpha^\gamma=\alpha^{p^n k}\equiv \alpha^k\mod p,
$
whence $s>0$, $m=s+n$, $\gamma_{i+1}=\langle a_0^{p^{i s}}, b_0^{p^{i s}}, c_0^{p^{(i-1)s}}\rangle$, $i\geq 1$,
and $c=\lceil m/s\rceil+1$.

$\bullet$ Case 1, $\alpha\equiv -1\mod p$, $h_0>n\geq 0$. Then $c=3$, with
$$\gamma_{2}=\langle a_0^{p^{h_0}}, b_0^{p^{h_0}}, c_0\rangle,\; 
\gamma_{3}=\langle a_0^{p^{m}}, b_0^{p^{m}}, c_0^{p^{h_0}}\rangle,\;
\gamma_{4}=1.
$$

$\bullet$ Case 1, $\alpha\equiv -1\mod p$, $h_0\leq n$. Then 
$\gamma_{i+1}=\langle a_0^{p^{i h_0}}, b_0^{p^{i h_0}}, c_0^{p^{(i-1)h_0}}\rangle$, $i\geq 1$,
so $c=\lceil m/h_0\rceil+1$.

$\bullet$ Case 2, $h>n\geq 0$, $h>1$. Then, $c=5$, with
$$\gamma_{2}=\langle a_0^{2^{h}}, b_0^{2^{h}}, c_0\rangle,\; 
\gamma_{3}=\langle a_0^{2^{m}}, b_0^{2^{m}}, c_0^{2^{h}}\rangle,\;
\gamma_{4}=\langle a_0^{2^{m+h}},c_0^{2^{m}}\rangle,\;
\gamma_{5}=\langle a_0^{2^{2m}}\rangle, \gamma_{6}=1.
$$

$\bullet$ Case 2, $1<h\leq n$. Then $\gamma_{i+1}=\langle a_0^{2^{i h}}, b_0^{2^{i h}}, c_0^{2^{(i-1)h}}\rangle$, $i\geq 1$,
so $c=\lceil m/h\rceil+2$.

$\bullet$ Case 2, $h=1$, $t\neq -1$, and $n>0$. Then $\gamma_{i+1}=\langle a_0^{2^{i}}, b_0^{2^{i}}, c_0^{2^{(i-1) }}\rangle$, $i\geq 1$,
so $c=2m+1-n$.

$\bullet$ Case 2, $h=1$, $t\neq -1$, and $n=0$. Then $c=3$ by \cite[Proposition 8.2]{MS}.

$\bullet$ Case 2, $\alpha=-1$. Then $c=3$ by \cite[Proposition 8.2]{MS}.

$\bullet$ Case 3, $\alpha\equiv 1\mod 3$. Then $c=7$ by \cite[Theorem 8.1]{MS}.

$\bullet$ Case 3, $\alpha\equiv -1\mod 3$. Then $c=3$, with
$\gamma_{2}=\langle a_0^{3}, c_0\rangle,\; 
\gamma_{3}=\langle a_0^{3}\rangle,\;\gamma_{4}=1.
$
\end{theorem}

\section{Derived length of $W$}

Note that $[W,W]=\langle X^{\alpha-1}, Y^{\alpha-1}, [X,Y]\rangle$. Clearly, the right hand side is contained in $[W,W]$.
On the other hand, $\langle X^{\alpha-1}, Y^{\alpha-1}, [X,Y]\rangle$ is normal in $W$, as it is clearly normalized by $Z$,
and $\langle X^{\alpha-1}, Y^{\alpha-1}, [X,Y]\rangle$ is a normal subgroup of $\langle X,Y\rangle$, since
$\langle A^{\alpha-1}, B^{\alpha-1}, [A,B]\rangle$ is a normal subgroup of $G$, being the kernel of 
$G\to C_{|\alpha-1|}\times C_{|\alpha-1|}$ (to see that the kernel is contained in $\langle A^{\alpha-1}, B^{\alpha-1}, [A,B]\rangle$,
use that $G=\langle A\rangle \langle B\rangle \langle [A,B]\rangle$, as ensured by 
\cite[Lemma 6.1]{MS}). 

As $[W,W]$ is contained in the finite nilpotent group $\langle X,Y\rangle$, we see that 
$[W,W]$ is the direct product of all $[W,W]_p=\langle X,Y\rangle_p\cap [W,W]$,
as $p$ runs through the positive prime factors of $\alpha^\gamma-1$. Thus, the derived
length of $W$, say $\ell(W)$, satisfies $\ell(W)=1+\max\{\ell([W,W]_p):p\mid (\alpha^\gamma-1)\}$.

Fix a prime number $p$ dividing $\alpha^\gamma-1$.
Then $[W,W]_p$ is the image of the restriction to $[W,W]=\langle X^{\alpha-1}, Y^{\alpha-1}, [X,Y]\rangle$ 
of projection map $\langle X,Y\rangle\to \langle X,Y\rangle_p$.
Recalling that $h=v_p(\alpha-1)$, we have
$$
[W,W]_p=\langle x^{p^h}, y^{p^h}, [x,y]\rangle\cong \langle a_0^{p^h}, b_0^{p^h}, [a_0,b_0]\rangle,
$$
in the notation of Section \ref{s2}, where $\langle a_0,b_0\rangle=G_p/N_p$. 

If $h=0$ we have $[W,W]_p=G_p/N_p$ and $\ell([W,W]_p)=\ell(G_p/N_p)$.

If $h>0$, the presentation of $G_p/N_p$ given in Section \ref{s2} shows that
$$
\langle a_0^{p^h}, b_0^{p^h}, [a_0,b_0]\rangle=[G_p/N_p, G_p/N_p],
$$
since $p\nmid kq$ and $h=v_p(\alpha-1)>0$. Thus $\ell([W,W]_p)=\ell(G_p/N_p)-1$ in this case.

(All of the above works even when $\alpha\neq\beta$, in which case $[W,W]=\langle X^{\alpha-1}, Y^{\beta-1}, [X,Y]\rangle$.)

We can now easily compute $\ell(G_p/N_p)$ using the the commutator formulas from Sections \ref{s4}, 
\ref{s5}, and \ref{s6}, together with Theorems \ref{tres}, \ref{main1}, \ref{hzero}, \ref{hzero2}, \ref{main2}, \ref{main3}, and \ref{poop2}.

Suppose first $h>0$. In Cases 1 and 2, except when $p=2$, $h=1$, and $n=0$, we have
$$
\gamma^{(1)}(G_p/N_p)=\langle a_0^{p^{h}}, b_0^{p^{h}}, c_0\rangle,
\gamma^{(2)}(G_p/N_p)=\langle a_0^{p^{m+h}}, b_0^{p^{m+h}}, c_0^{p^{2h}}\rangle,
\gamma^{(3)}(G_p/N_p)=1.
$$
and $\ell(G_p/N_p)=3$. When $p=2$, $h=1$, and $n=0$, we have $\ell(G_2/N_2)=\ell(G_2)=\ell(Q_{16})=2$.
In Case 3, we have $\alpha\equiv 1\mod 3$ and $n=0$,
as explained in Section \ref{s7}, so $N_3=1$. The commutator formulas from Section \ref{s6} give $\ell(G_3)=3$.

Suppose next $h=0$. This can only occur in Cases 1 and 3.
In Case 3, we readily see that $\ell(G_3/N_3)=2$. In Case 1, if $\alpha\not\equiv -1\mod p$, then $G_p/N_p\cong\mathrm{Heis}(\Z/p^m\Z)$,
so $\ell(G_p/N_p)=2$. In Case 1, if $\alpha\equiv -1\mod p$, then the terms of the derived
series of $G_p/N_p$ are 
$$
\gamma^{(1)}(G_p/N_p)=\langle a_0^{p^{m}}, b_0^{p^{m}}, c_0\rangle,
\gamma^{(2)}(G_p/N_p)=1,
$$
so $\ell(G_p/N_p)=2$. We have shown the following result.

\begin{theorem}\label{solv} We have $\ell(W)=2$ if $\alpha=-1$ and $\gamma$ is odd, or $\alpha=3$ and $\gamma=1$; else
$\ell(W)=3$.
\end{theorem}




\end{document}